\documentclass[a4paper,reqno]{amsart}
\usepackage[a4paper, total={7in, 8.5in}]{geometry}
\usepackage{biblatex}
 \usepackage{amsaddr}
\renewbibmacro{in:}{%
  \ifentrytype{article}{}{\printtext{\bibstring{in}\intitlepunct}}}
\usepackage{amsfonts}
\usepackage{amsthm}
\usepackage{amsmath}
\usepackage{datetime}
\usepackage{comment}
\usepackage{enumitem}
\usepackage[dvipsnames]{xcolor}
\usepackage{tikz}
\usetikzlibrary{plotmarks}
\usetikzlibrary{decorations.pathreplacing}
\numberwithin{equation}{section}
\usetikzlibrary{plotmarks}
\usepackage{tkz-graph}
\usepackage[font=small,labelfont=bf]{caption}
\usetikzlibrary{shapes,arrows,automata,decorations.pathreplacing,angles,quotes,calc}
\usepackage{tkz-euclide}
\usepackage{subcaption}

\newtheorem{theorem}{Theorem}[section]

\newtheorem{lemma}[theorem]{Lemma}
\newtheorem{remark}[theorem]{Remark}
\newtheorem{definition}[theorem]{Definition}
\newtheorem{proposition}[theorem]{Proposition}
\numberwithin{figure}{section}

\title{Asymptotics for multiple $q$-orthogonal polynomials from the Riemann Hilbert Problem}

\author{Tomas Lasic Latimer}
\address{University of California Santa Cruz}
\email{tlasicla@ucsc.edu}

\date{}
\begin{document}
\begin{abstract}
    We deduce the asymptotic behaviour of a broad class of multiple $q$-orthogonal polynomials as their degree tends to infinity.  We achieve this by rephrasing multiple $q$-orthogonal polynomials as part of a solution to a Riemann Hilbert Problem (RHP). In particular, we study multiple $q$-orthogonal polynomials of the first kind (see \cite{postelmans2005multiple}), which are Type II orthogonal polynomials with two orthogonality conditions given by Equation \eqref{the general weights}. Using $q$-calculus we obtain detailed asymptotics for these polynomials from the RHP. This class of polynomials was studied in part due to their connection to the work of \cites{POSTELMANS2007119}{postelmans2005multiple}, concerning the irrationality of $\zeta_q(1)$ and $\zeta_q(2)$. \\
    
\end{abstract}

\maketitle

\tableofcontents

\section{Introduction}
\subsection{Motivation}
The motivation of this paper is twofold. First, we wish to investigate how the theory presented in \cite{NJTLconst}, for $q$-orthogonal polynomials with a single orthogonality condition, extends to multiple $q$-orthogonal polynomials. In \cite{NJTLconst} we were able to determine the strong asymptotics for $q$-orthogonal polynomials using the RHP and $q$-calculus. As far as we know, these results are the most accurate description of the asymptotics of $q$-orthogonal polynomials found in the literature. Interestingly, in \cite{NJTLconst} we observed that if polynomials $P^{(1)}_n(x)$ and $P^{(2)}_n(x)$ were such that their weights satisfied
\[ \lim_{ x\to 0} \frac{w^{(1)}(x)}{w^{(2)}(x)} = 1,\]
then $P^{(1)}_n(x)$ and $P^{(2)}_n(x)$ had the same asymptotic behaviour (to second order) as $n\to\infty$ (there are some additional minor details to consider for which the reader can refer to \cite{NJTLconst}). Note that in the language of \cite{baik2007discrete}, $0$ is an accumulation point of the $q$-lattice. In the present paper, we successfully extend the theory presented in \cite{NJTLconst} to the case of multiple $q$-orthogonal polynomials and once again find more accurate asymptotics than previously known in the literature (see \cite{postelmans2005multiple}). Furthermore, we observe that in the case of multiple orthogonal polynomials it is no longer sufficient to conclude that if two different sets of weights have similar behaviour at $x=0$, then the second order asymptotics of the corresponding multiple orthogonal polynomials agree as $n\to\infty$. This phenomenon is discussed further in the conclusion of the paper.\\

We also wished to study the asymptotics of multiple $q$-orthogonal polynomials due to their appearance in analysing the $q$-zeta function \cite{POSTELMANS2007119}. In particular, a more accurate estimation of the behaviour as $n\to \infty$ of the multiple $q$-orthogonal polynomials studied in \cite{POSTELMANS2007119}, would provide better bounds on the measure of irrationality of 1, $\zeta_q(1)$ and $\zeta_q(2)$. The desire to apply our results to \cite{POSTELMANS2007119} motivated the choice of polynomials studied. This is why we have chosen to investigate the asymptotics of multiple $q$-orthogonal polynomials of the first kind (see Remark \ref{why first?}), with weights given by
\begin{subequations}\label{the general weights}
    \begin{eqnarray}
    w_1(x) = x^\alpha \omega(x)d_qx,\\
    w_2(x) = x^\beta \omega(x)d_qx,
\end{eqnarray}
\end{subequations}
which satisfy the constraint
\begin{equation}\label{the weight constraint}
    |\omega(q^{2n})-1| = \mathcal{O}(q^{2n}), 
\end{equation}
as $n\to \infty$. In order to guarantee the existence of such polynomials we also require that $(w_1(x)d_qx,w_2(x)d_qx)$ is an AT system (see \cite[Chapter 23]{ismail2005classical}), note that this implies that $\alpha - \beta \notin \mathbb{Z}$. Moreover, to conduct our asymptotic analysis we will require the following additional restrictions on $w_1(x)$ and $w_2(x)$: $\alpha \notin \mathbb{Z}$, $\beta \notin \mathbb{Z}$. This additional restriction is necessary for the asymptotic analysis conducted in this paper. However, the condition $\alpha \notin \mathbb{Z}$, $\beta \notin \mathbb{Z}$ is not needed for the existence of multiple $q$-orthogonal polynomials and their statement as solutions to a RHP. The restriction that $\alpha \notin \mathbb{Z}$ and $\beta \notin \mathbb{Z}$ allows us to construct the functions $h_\alpha(z)$ and $h_\beta (z)$ (see Lemma \ref{h lemma}) used to carry out our asymptotic analysis. The existence of $h_\alpha(z)$ and $h_\beta(z)$ can be seen as a consequence of the existence of a solution to the RHP that is analytic at $z=0$, which is critical to our asymptotic analysis. \\

Note that the constraint given by Equation \eqref{the weight constraint} is fundamental to our main results. It describes the behaviour of the weight near the accumulation point $z=0$ of the $q$ lattice. Our main theorems essentially state that if two weights are `close' enough (where `close' is defined by Equation \eqref{the weight constraint}) then the corresponding multiple orthogonal polynomials share similar asymptotic behaviour as their degree tends to infinity. If Equation \eqref{the weight constraint} is not satisfied, then following the RHP transformations presented in this paper, we do not determine a RHP with jump close to the identity. Thus, we do not find the same asymptotic behaviour. This point is also discussed in the conclusion of this paper.\\

Further work is needed to use the results presented here, see Theorem \ref{Main theorem} for our main results, to gain a better understanding of the measure of irrationality of 1, $\zeta_q(1)$ and $\zeta_q(2)$.

\begin{remark}\label{why first?}
    Two different types of multiple little $q$-Jacobi polynomials were introduced in  \cite{postelmans2005multiple}, both of which are type II multiple orthogonal polynomials \cite{WVA99}. These were multiple little $q$-Jacobi polynomials of the first kind, with weights $w(x;a_j,b|q)d_qx$ and, multiple little $q$-Jacobi polynomials of the second kind, with weights $w(x;a,b_j|q)d_qx$, where 
    \[ w(x;a,b|q) = x^a\frac{(qx;q)_\infty}{(q^{b+1}x;q)_\infty}.\]
    Note that the weights we study in this paper are generalisations of multiple little $q$-Jacobi polynomials of the first kind. The asymptotic results we find do not translate readily to multiple little $q$-Jacobi polynomials of the second kind, which will need a different approach. In \cite{POSTELMANS2007119} they use the limit of multiple little $q$-Jacobi polynomials of the first kind, as $a_1\to a_2$, as the beginning of their analysis of the rationality of $\zeta_q(1)$ and $\zeta_q(2)$.
\end{remark}

The author wishes to extend special thanks to Prof. Walter Van Assche, who motivated this study and provided valuable discussion.
\subsection{Notation}\label{notation}
For completeness, we recall some well known definitions and notation from the calculus of $q$-differences. These definitions can be found in \cite{Ernst2012}. Throughout the paper we will assume $q \in \mathbb{R}$ and $0<q<1$. We also recall some nomenclature from multiple orthogonal polynomial theory.

\begin{definition}
We define the Pochhammer symbol $(a;q)_{n}$, and un-normalised Jackson integrals as follows.
\begin{enumerate}

\item The Pochhammer symbol $(a;q)_{n}$ is defined as
\begin{equation*}
    (a;q)_{n} = \prod_{j=0}^{n-1}(1-aq^{j}) \,.
\end{equation*}
It is convention that $(a;q)_0:=1$.

\item The un-normalised Jackson integral from 0 to 1 is defined as

\begin{equation}\label{Jackson}
    \int_{0}^{1}f(x)d_{q}x = \sum_{k=0}^{\infty} f(q^{k})q^{k} \,.
\end{equation}
We will use the un-normalised Jackson integral to define the orthogonality conditions satisfied by multiple $q$-orthogonal polynomials in this paper.
\end{enumerate}
\end{definition}

We recall a slightly modified definition of an \textit{appropriate} curve $\Gamma$ to that given in \cite[Defintion 1.2]{qRHP}. 

\begin{definition}\label{admissable}
A positively oriented Jordan curve $\Gamma$ in $\mathbb C$ with interior $\mathcal D_-\subset\mathbb C$ and exterior $\mathcal D_+\subset\mathbb C$ is called {\em appropriate} if 
\[
q^k\in \begin{cases}
&\mathcal D_- \quad {\rm if}\, k\ge 0,\\
&\mathcal D_+ \quad {\rm if}\, k< 0.
\end{cases}
\]
An example of an {\em appropriate} curve is illustrated in Figure \ref{RHP-figure}.
\end{definition}

\begin{definition}\label{scaled curve}
 We define $\Gamma_\lambda$ $(\lambda \neq 0)$ as the curve $\Gamma$ scaled such that the modulus of the points on it are multiplied by $
\lambda$ $($i.e. if $\Gamma$ were the unit circle, $\Gamma_{\lambda}$ would be a circle with radius $\lambda$$)$.   
\end{definition}

\begin{definition}\label{multiple ortho def}
    We define multiple $q$-orthogonal polynomials as the set of monic polynomials $\{ P_{n,m}(x)\}_{n,m=0}^\infty$, 
which satisfy the orthogonality relation
\begin{subequations}\label{multiple q ortho def}
\begin{eqnarray}
\int_{0}^1 P_{n,m}(x)x^k\mu_1(x) d_qx &=& 0,\,\mathrm{for}\,k<n ,\\
\int_{0}^1 P_{n,m}(x)x^j\mu_2(x) d_qx &=& 0,\,\mathrm{for}\,j<m .
\end{eqnarray}
\end{subequations}
Note that if $(\mu_1(x)d_qx,\mu_2(x)d_qx)$ is an AT system (see \cite[Chapter 23]{ismail2005classical}) then $P_{n,m}(x)$ exists and is of degree $n+m$.
\end{definition}

\begin{definition}\label{L2 def}
We define the constants $\gamma_1^{(n-1,n)}$, $\gamma_2^{(n,n-1)}$ and $\gamma_1^{(n,n)}$ by the equations:
\begin{subequations}\label{gamma def}
    \begin{eqnarray}
        \gamma_1^{(n-1,n)} &=& \int_{0}^1 P_{n-1,n}(x)x^{n-1} \mu_{1}(x) d_qx,\\
        \gamma_2^{(n,n-1)} &=& \int_{0}^1 P_{n,n-1}(x)x^{n-1} \mu_{2}(x) d_qx,\\
        \gamma_1^{(n,n)} &=& \int_{0}^1 P_{n,n}(x)x^n \mu_{1}(x) d_qx,
    \end{eqnarray}
\end{subequations}
where $\{P_{n,m}(x)\}_{n,m=0}^\infty$ is the set of multiple $q$-orthogonal polynomials defined in Definition \ref{multiple ortho def}.
\end{definition}

\subsection{Main Results}
The main results of this paper are collected in the theorem below. This theorem is proved in Section \ref{Uni section}.
\begin{theorem}\label{Main theorem}
Given the set of multiple $q$-orthogonal polynomials $($see Definition \ref{multiple ortho def}$)$ with weights $\mu_1(x)=w_1(x)$ and $\mu_2(x)=w_2(x)$, where $w_1(x)$ and $w_2(x)$ are defined in Equation \eqref{the general weights} and $\omega(x)$ satisfies Equation \eqref{the weight constraint}, then the following identities hold:
      \begin{enumerate}
    \item $\lim_{n\to\infty} q^{n(1-2n)}P_{n,n}(q^{2n}z) = F_1(z)$, where 
     \[F_1(z)=\frac{1}{\Omega^{(1)}_\infty}\sum_{j=0}^\infty \frac{(-1)^jq^{j(j+1)/2}}{(q^{\alpha+1};q)_j(q^{\beta+1};q)_j(q;q)_j}z^j,\]
     is an entire function and $\Omega^{(1)}_\infty$ is a non-zero constant defined in Equation \eqref{omega 1 def}. Furthermore, convergence is uniform in any disc $|z|<R$.
        \item For $k\in\mathbb{N}$ larger than some critical value $k_c$ (which is independent of $n$) it holds that
        \begin{multline}\nonumber
            P_{n,n}(q^{2n-2}q^{-k})(q^{2n+1}q^{-k};q)_\infty = \\ \mathcal{O}\Big(\max\big(q^{n}P_{n,n}(q^{-1}q^{2n-k}), q^{k}P_{n,n}(q^{-1}q^{2n-k}), q^{2n}P_{n,n}(q^{2n-k}),q^{k}P_{n,n}(q^{2n-k})\big)\Big).
        \end{multline}

        \item $\lim_{n\to\infty} q^{n(3-2n)}P_{n,n-1}(q^{2n}z) = F_2(z)$, where $F_2(z)=\lambda_1F_1(qz)$, and $\lambda_1$ is a non-zero constant. Furthermore, convergence is uniform in any disc $|z|<R$.
        \item For $k\in\mathbb{N}$ larger than some critical value $k_c$ (which is independent of $n$) it holds that
        \begin{multline}\nonumber 
            P_{n,n-1}(q^{2n-2}q^{-k})(q^{2n+1}q^{-k};q)_\infty = \\ \mathcal{O}\Big(\max\big(q^{n}P_{n,n-1}(q^{-1}q^{2n-k}), q^{k}P_{n,n-1}(q^{-1}q^{2n-k}), q^{2n}P_{n,n-1}(q^{2n-k}),q^{k}P_{n,n}(q^{2n-k})\big)\Big).
        \end{multline}

        \item $\lim_{n\to\infty} q^{n(3-2n)}P_{n-1,n}(q^{2n}z) = \lambda_3 F_2(z)$, for some non-zero constant $\lambda_3$. Furthermore, convergence is uniform in any disc $|z|<R$. Note that up to linear scaling both $P_{n,n-1}(zq^{2n})$ and $P_{n-1,n}(zq^{2n})$ approach the same function $F_2(z)$. 
       \item For $k\in\mathbb{N}$ larger than some critical value $k_c$ (which is independent of $n$) it holds that
        \begin{multline}\nonumber
            P_{n-1,n}(q^{2n-2}q^{-k})(q^{2n+1}q^{-k};q)_\infty =  \\\mathcal{O}\Big(\max\big(q^{n}P_{n,n-1}(q^{-1}q^{2n-k}), q^{k}P_{n,n-1}(q^{-1}q^{2n-k}), q^{2n}P_{n,n-1}(q^{2n-k}),q^{k}P_{n,n}(q^{2n-k})\big)\Big).
        \end{multline}
 
    \end{enumerate}
Furthermore, the norms $\gamma_1^{(n,n)}$, $\gamma_1^{(n-1,n)}$ and $\gamma_2^{(n,n-1)}$, defined in Equation \eqref{gamma def} display the following asymptotic behaviour as $n\to\infty$.
    \begin{enumerate}
    \item $\lim_{n\to\infty}q^{-n(3n+\alpha+\beta)}\gamma_1^{(n,n)} = C_1$, where $C_1$ is some non-zero constant independent of $n$.
    \item $\lim_{n\to\infty}q^{-n(3n+\alpha+\beta-3)}\gamma_1^{(n-1,n)} = C_2$, where $C_2$ is some non-zero constant independent of $n$.
    \item $\lim_{n\to\infty}q^{-n(3n+\alpha+\beta-3)}\gamma_2^{(n,n-1)} = C_3$, where $C_3$ is some non-zero constant independent of $n$.
    \end{enumerate}
    The constants $C_i$ can be deduced from the arguments presented in Section \ref{Sect 4}. They are the ratio of functions which can be written as converging power series.
\end{theorem}

\begin{remark}
   Points (2), (4) and (6) in the first part of Theorem \ref{Main theorem} tell us that, past some critical value $k_c$,  $P_{n,n}(q^{2n}q^{-k})$ $($or conversely $P_{n\pm1,n\pm1})$ shrinks very rapidly as $k$ increases up to the point $k = 2n+1$ $(k\in\mathbb{N})$. In particular, Theorem \ref{Main theorem} tells us that $P_{n,n}(q^{2n-2}q^{-k})$ is much smaller than $P_{n,n}(q^{2n-2}q^{-k+1})$ or $P_{n,n}(q^{2n-2}q^{-k+2})$. Note that the factor of $(q^{2n+1}q^{-k};q)_\infty$ is bounded above by $1$ and below by $(q;q)_\infty$ for $k<2n+1$, and so it does not affect our order estimate of $P_{n,n}(q^{2n-2}q^{-k})$ for $k<2n+1$.  However, at $k=2n+1$, $(q^{2n+1}q^{-k};q)_\infty=0$. Furthermore, note that by following the arguments presented in Section \ref{Sect 3}, $k_c$ (which is independent of $n$) is a function of $\alpha$, $\beta$ and $q$ that can be numerically determined by considering Equation \eqref{f mat diff}. For example a possible (sub-optimal) choice is $k_c=|\alpha|+|\beta|+2+\ln_q(10)$.
    
\end{remark} 

\subsection{Outline}
The structure of the paper is detailed below:
\begin{enumerate}
    \item Adopting the nomenclature of  \cite{postelmans2005multiple}, we first study the simplest case of multiple little $q$-Jacobi polynomials of the first kind, with weights
    \begin{subequations}\label{the weights}
        \begin{eqnarray}
    w_1(x) = x^\alpha (xq;q)_\infty,\\
    w_2(x) = x^\beta (xq;q)_\infty.
\end{eqnarray}
    \end{subequations}

    Note that $(w_1(x)d_qx, w_2(x)d_qx)$ is an AT system and the polynomials are guaranteed to exist if $\alpha-\beta \notin \mathbb{Z}$ \cite[Chapter 23.6.5.1]{ismail2005classical}. We ascertain a corresponding RHP whose solution, $Y_n(z)$, will be given in terms of these multiple $q$-orthogonal polynomials. This then allows us to determine a Lax pair where one equation describes an iteration in the degree $n$, and the other, a $q$-difference equation in the complex variable $z$.
    \item We study the corresponding matrix $q$-difference equation in detail. This allows us to write down a third order $q$-difference equation satisfied by the entries of $Y_n(z)$, which in turn gives us valuable insight into the behaviour of the entries of $Y_n(z)$.
    \item Having used the $q$-difference equation to describe the asymptotic behaviour of $Y_n(z)$ for the simplest case of weights given by Equation \eqref{the weights}, we extend our results to the general class of weights satisfying Equation \eqref{the general weights}.
    \item We conclude with a discussion of the results found, drawing connections to earlier work done in the case of $q$-orthogonal polynomials with a single orthogonality condition. We also discuss possible directions for future research.
\end{enumerate}

\section{Constructing the RHP}
\subsection{Preliminary discussion}\label{Prelim sect}
In an earlier paper \cite[Lemma A.1]{jt2023}, we showed that there does not exist a function $h(z)$, such that $h(qz)=h(z)$ and $h(z)$ is meromorphic in $\mathbb{C}\setminus \{0\}$ with simple poles located exclusively at $z=q^k$, for $k\in\mathbb{Z}$. It follows that an equivalent statement can be made for functions $\hat{h}(z)$ which satisfy $\hat{h}(qz)=q^{j}\hat{h}(z)$ for $j\in\mathbb{Z}$. If such a function where to exist then $h(z) = \hat{h}(z)z^{-j}$ would violate \cite[Lemma A.1]{jt2023}. However, let $-1<\alpha<0$, and consider the function
\begin{equation}
    h_\alpha (z) = \sum_{k=-\infty}^\infty \frac{q^{k(1+\alpha)}}{z-q^k}.
\end{equation}
This sum converges for all $z \in \mathbb{C} \setminus \{q^{k}: k \in \mathbb{Z} \}\cup \{0\}$ and furthermore, direct calculation shows that $h_\alpha(qz) = q^\alpha h_\alpha(z)$. Suppose instead that $0<\alpha<1$, and consider the function
\begin{equation}
    h_\alpha (z) = \sum_{k=-\infty}^\infty \frac{zq^{k\alpha}}{z-q^k}.
\end{equation}
Once again, this sum converges for all $z \in \mathbb{C} \setminus \{q^{k}: k \in \mathbb{Z} \}\cup \{0\}$ and direct calculation shows $h_\alpha(qz) = q^\alpha h_\alpha(z)$. In fact, one can construct a meromorphic function in $\mathbb{C}\setminus \{0\}$ with simple poles located exclusively at $z=q^k$ for $k\in\mathbb{Z}$, and which satisfies $h_\alpha(qz) = q^\alpha h_\alpha(z)$ for any $\alpha \notin \mathbb{Z}$. Furthermore this function is unique up to scalar multiplication. We prove this statement in the following lemma.
\begin{lemma}\label{h lemma}
    There exists a unique (up to scalar multiplication) meromorphic function in $\mathbb{C}\setminus \{0\}$ with simple poles located exclusively at $z=q^k$ for $k\in\mathbb{Z}$, and which satisfies $h_\alpha(qz) = q^\alpha h_\alpha(z)$ for any $\alpha \notin \mathbb{Z}$.
\end{lemma}
\begin{proof}
    Consider the function
    \[ \tilde{h}_1(z) = (qz;q)_\infty (z^{-1};q)_\infty .\]
    By direct calculation one can show $\tilde{h}_1(q^{-1}z) = -z\tilde{h}_1(z)$. Suppose that $h_\alpha(z)$ exists and let us define
    \begin{equation}\label{halha def}
        \tilde{h}_2(z) = h_{\alpha}(z)\tilde{h}_1(z). 
    \end{equation} 
    It follows that $\tilde{h}_2(z)$ is analytic everywhere in $\mathbb{C}\setminus \{0\}$ and satisfies the difference equation
    \begin{equation}\label{f2 difference}
        \tilde{h}_2(q^{-1}z) = -zq^{\alpha}\tilde{h}_2(z).
    \end{equation} 
    As $\tilde{h}_2(z)$ is analytic in $\mathbb{C}\setminus \{0\}$ it can be written in terms of a Laurent series. Substituting this into Equation \eqref{f2 difference} we find
    \begin{equation}\label{f2 series}
        \tilde{h}_2(q^{-1}z) = \sum_{j=-\infty}^\infty c_jq^{-j}z^j = -\sum_{j=-\infty}^\infty c_jq^{\alpha}z^{j+1} .
    \end{equation}
    Equating coefficients of $z$ in Equation \eqref{f2 series} gives us the recurrence relation
    \[ c_jq^{-j} = c_{j-1}q^{\alpha}.\]
    The above recurrence relation clearly has a unique solution up to scalar multiplication (i.e. given $c_0=1$ the solution is unique). Furthermore, direct calculation shows that the function $\hat{h}_2(z)$ defined as
    \[ \hat{h}_2(z) = (q^{\alpha+1}z;q)_\infty (q^{-\alpha}z^{-1};q)_\infty ,\]
    satisfies Equation \eqref{f2 difference}. Hence, $\hat{h}_2(z)=\lambda \tilde{h}_2(z)$ for some scalar constant $\lambda$. Thus, the zeros of $\tilde{h}_2(z)$ do not lie at $z=q^{k}$ and it follows $ h_{\alpha}(z)$ defined by Equation \eqref{halha def} has simple poles at $z=q^{k}$ as required.
\end{proof}

We will construct our RHP using these $h_\alpha(z)$ functions. 
\begin{remark}
    One could instead use a residue condition, similar to Baik et al. \cite{baik2007discrete}, in order to construct the RHP. This would avoid the need for having $h_\alpha(z)$ and hence allow for $\alpha \in\mathbb{Z}$. The condition that $\alpha \notin \mathbb{Z}$ is required to proceed with our asymptotic argument but it is not needed to define the multiple orthogonal polynomials themselves, which are well defined for any $\alpha$, $\beta$ such that $\alpha\neq \beta \mod \mathbb{Z}$.
\end{remark}

\subsection{The RHP and Lax Pair}
The $3\times3$ matrix function $Y_n(z)$ is a solution to RHP I if it satisfies the following three conditions:
\begin{enumerate}[label={{\rm (\roman *)}}]
\begin{subequations}
\item $Y_n(z)$ is analytic on $\mathbb{C}\setminus \Gamma$.
\item $Y_n(z)$ has continuous boundary values $Y_{n,-}(s)$ and $Y_{n,+}(s)$ as $z$ approaches $s\in\Gamma$ from $\mathcal D_-$ and $\mathcal D_+$ respectively, where $\Gamma$ is an appropriate curve (see Definition \ref{admissable}). Furthermore, $Y_{n,-}(s)$ and $Y_{n,+}(s)$ are related by the jump condition:
\begin{gather} \label{jump cond}
Y_{n,+}(s) =
Y_{n,-}(s)
\begin{pmatrix}
1 &
w(s)h_\alpha(s) & w(s)h_\beta(s) \\
0 & 1 & 0 \\
0 & 0 & 1 
\end{pmatrix}, \; s\in \Gamma ,
\end{gather}
$w(s) = (sq;q)_\infty$ and $h_\alpha(s)$ and $h_\beta(s)$ are as discussed in Section \ref{Prelim sect}.
\item $Y_n(z)$ satisfies
\begin{gather} \label{decay cond}
Y_n(z)\begin{pmatrix}
z^{-2n} & 0 & 0\\
0 & z^{n} & 0 \\
0 & 0 &z^{n}
\end{pmatrix}
=
I + \mathcal{O}\left( \frac{1}{z} \right), \text{ as }\ |z| \rightarrow \infty.
\end{gather}

\end{subequations}
\end{enumerate}

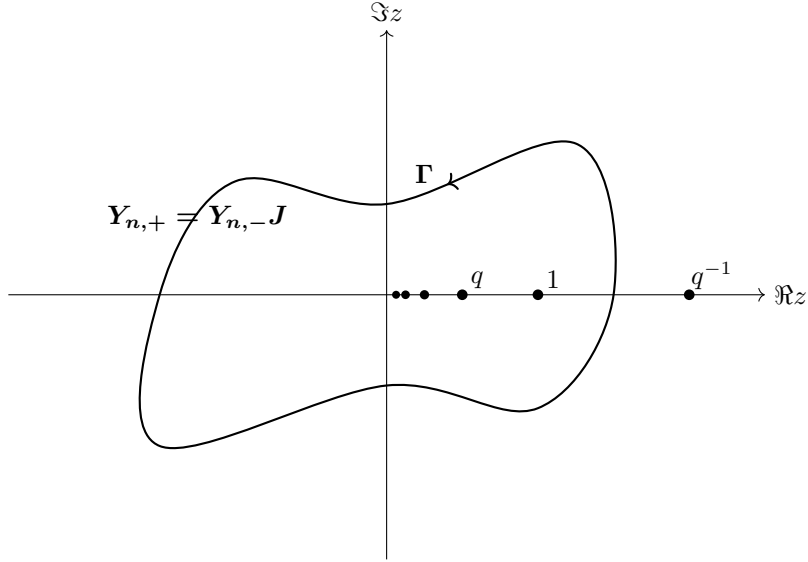
\begin{figure}[ht]
\centering

\begin{tikzpicture}[scale=1]

\draw[->] (-5,0)--(5,0) node[right]{$\Re{z}$};

\draw[->] (0,-3.5)--(0,3.5) node[above]{$\Im{z}$};

\tikzstyle{star}  = [circle, minimum width=4pt, fill, inner sep=0pt];

\tikzstyle{starsmall}  = [circle, minimum width=3.5pt, fill, inner sep=0pt];

\tikzstyle{vstarsmall}  = [circle, minimum width=3.25pt, fill, inner sep=0pt];

\tikzstyle{vvstarsmall}  = [circle, minimum width=3pt, fill, inner sep=0pt];

\node[star]     (q0) at ({2*1},{0} ) {};
\node     at ($(q0)+(0.2,0.2)$) {$1$};

\node[star]     (qm) at ({2*2)},{0} ) {};
\node     at ($(qm)+(0.3,0.25)$) {$q^{-1}$};



\node[star]     (iqm) at ({0.5*2},{0} ) {};
\node     at ($(iqm)+(0.2,0.2)$) {$q$};

\node[starsmall]     (iqm) at ({0.25*2},{0} ) {};
\node[vstarsmall]     (iqm) at ({0.125*2},{0} ) {};
\node[vvstarsmall]     (iqm) at ({0.125*2*0.5},{0} ) {};

\draw [black,thick,decoration={markings, mark=at position 0.68 with {\arrow{<}}},postaction={decorate}] plot [smooth cycle,tension=0.6] coordinates {(3,0) (2,-1.5) (0,-1.2) (-3,-2) (-3,0) (-2,1.5) (0,1.2) (2.5,2) };  
\node[black]     at (0.5,1.6) {$\boldsymbol{\Gamma}$};
\node    at (-2.5,1) {$\boldsymbol{Y_{n,+} = Y_{n,-}J}$};
\end{tikzpicture}
\caption{Example of an \textit{appropriate} curve $\Gamma$.}
\label{RHP-figure}
\end{figure}

Following the arguments presented in \cite{qRHP} and the definition of multiple orthogonal polynomials one readily deduces that RHP I has the unique solution:
\begin{gather} \label{RHP sol 1}
Y_n(z) 
=
\begin{bmatrix}
   P_{n,n}(z) &
   \oint_\Gamma \frac{P_{n,n}(s)w(s)h_\alpha(s)}{z-s} ds&
   \oint_\Gamma \frac{P_{n,n}(s)w(s)h_\beta(s)}{z-s} ds \\
   \frac{1}{\gamma_1^{(n-1,n)}}P_{n-1,n}(z)&
    \frac{1}{\gamma_1^{(n-1,n)}}\oint_\Gamma \frac{P_{n-1,n}(s)w(s)h_\alpha(s)}{z-s} ds &
    \frac{1}{\gamma_1^{(n-1,n)}}\oint_\Gamma \frac{P_{n-1,n}(s)w(s)h_\beta(s)}{z-s} ds\\
    \frac{1}{\gamma_2^{(n,n-1)}}P_{n,n-1}(z) &
    \frac{1}{\gamma_2^{(n,n-1)}}\oint_\Gamma \frac{P_{n,n-1}(s)w(s)h_\alpha(s)}{z-s} ds&
    \frac{1}{\gamma_2^{(n,n-1)}}\oint_\Gamma \frac{P_{n,n-1}(s)w(s)h_\beta(s)}{z-s} ds
   \end{bmatrix},
\end{gather}
where $P_{n,m}(z)$ satisfies Equation \eqref{ortho 1}. 
\begin{subequations}\label{ortho 1}
\begin{eqnarray}
\int_{0}^1 P_{n,m}(x)x^kw_{1}(x) d_qx &=& 0,\,\mathrm{for}\,k<n ,\\
\int_{0}^1 P_{n,m}(x)x^jw_{2}(x) d_qx &=& 0,\,\mathrm{for}\,j<m ,
\end{eqnarray}
\end{subequations}
and $w_1(x)$ and $w_2(x)$ are given by Equation \eqref{the weights}. \\

Note that for $z\in\text{ext}(\Gamma)$ the solution to RHP I given by Equation \eqref{RHP sol 1} looks like,
\begin{gather} \label{RHP sol 2}
Y_n(z) 
=
\begin{bmatrix}
   P_{n,n}(z) &
   \sum_{k=0}^\infty \frac{P_{n,n}(q^k)w_1(q^k)q^k}{z-q^k} &
   \sum_{k=0}^\infty \frac{P_{n,n}(q^k)w_2(q^k)q^k}{z-q^k} \\
   \frac{1}{\gamma_1^{(n-1,n)}}P_{n-1,n}(z)&
    \frac{1}{\gamma_1^{(n-1,n)}}\sum_{k=0}^\infty \frac{P_{n-1,n}(q^k)w_1(q^k)q^k}{z-q^k} &
    \frac{1}{\gamma_1^{(n-1,n)}}\sum_{k=0}^\infty \frac{P_{n-1,n}(q^k)w_2(q^k)q^k}{z-q^k}\\
    \frac{1}{\gamma_2^{(n,n-1)}}P_{n,n-1}(z) &
    \frac{1}{\gamma_2^{(n,n-1)}}\sum_{k=0}^\infty \frac{P_{n,n-1}(q^k)w_1(q^k)q^k}{z-q^k}&
    \frac{1}{\gamma_2^{(n,n-1)}}\sum_{k=0}^\infty \frac{P_{n,n-1}(q^k)w_2(q^k)q^k}{z-q^k}
   \end{bmatrix}.
\end{gather}

\begin{remark}\label{could generalise remark}
    Combining the work of \cite[Section 5]{BK} and \cite{qRHP} one can readily pose a more general form of RHP 1 corresponding to $P_{n,m}(z)$ where $m$ is independent of $n$. We study the more specific case of $m=n$ as this pertains more closely the polynomials used in \cite{POSTELMANS2007119}. Furthermore, the simplification $m=n$ reduces the algebraic complexity of our asymptotic argument which helps highlight key points and the close connection to $q$-calculus. 
\end{remark}

\begin{remark}
    Using standard arguments in RHP analysis it follows that $\mathrm{det}(Y_n(z))=1$ $($see for example \cite{Deift1999strong},\cite{qRHP}$)$. 
\end{remark}

\begin{remark}
    Note that the orthogonality condition, Equation \eqref{ortho 1}, is with respect to the weights $w_1(x) = x^\alpha (xq;q)_\infty$ and $w_2 = x^\beta(xq;q)_\infty$. However, the jump condition, Equation \eqref{jump cond}, has off-diagonal entries $h^\alpha(s)(sq;q)_\infty$ and $h^\beta(s)(sq;q)_\infty$. This is because it is the residue of $h^{\alpha}(z)$ that allows us to write the Cauchy transform as the sum of the residues of $h^{\alpha}(z)(sq;q)_\infty$ $($compare Equations \eqref{RHP sol 1} and \eqref{RHP sol 2}$)$, which in turn corresponds to the discrete Jackson integral given in Equation \eqref{ortho 1}. 
\end{remark}

As $Y_n(z)$ is a solution to RHP I we will now show that it satisfies a Lax pair. First, we look at an iteration in the degree $n$. Define,
\begin{equation}
    V_n(z) = Y_{n+1}(z)Y_n(z)^{-1}.
\end{equation}
As solutions to RHP I, for $n+1$ and $n$ respectively, both the entries of $ Y_{n+1}(z)$ and $Y_n(z)$ are analytic in $\mathbb{C}\setminus \Gamma$. Furthermore, as det($Y_n(z)$)=1 it follows that the entries of $Y_n(z)^{-1}$ are analytic in $\mathbb{C}\setminus \Gamma$. Hence, $V_n(z)$ is analytic in $\mathbb{C}\setminus \Gamma$. The jump for $V_n(z)$ across $\Gamma$ is given by 
\begin{eqnarray*}
    V_n^-(z)^{-1}V_n^{+}(z) &=& (Y_{n+1}^-(z)Y_n^-(z)^{-1})^{-1}(Y_{n+1}^+(z)Y_n^+(z)^{-1}),\\
    &=& Y_n^-(z)Y_{n+1}^-(z)^{-1}Y_{n+1}^+(z)Y_n^+(z)^{-1},\\
    &=& Y_n^-(z)J Y_n^+(z)^{-1},\\
    &=& Y_n^+(z) Y_n^+(z)^{-1},\\
    &=& I,
\end{eqnarray*}
where $J$ is the jump matrix defined by Equation \eqref{jump cond}. Thus, the entries of $V_n(z)$ are analytic everywhere. Their behaviour as $z\to\infty$ can be found using Equation \eqref{decay cond} allowing one to determine the polynomial entries of $V_n(z)$ whose coefficients are related to the series expansion of $Y_n(z)$ at $z\to\infty$. This analysis provides us with the well known four term recurrence relation satisfied by multiple orthogonal polynomials \cite{MultRec}.\\

We will now study the $q$-difference component of the Lax pair. Let $Y_n(z)$ be a solution to RHP I with an appropriate curve $\Gamma$. As the jump matrix, Equation \eqref{jump cond}, has analytic entries in ext($\Gamma$) it follows by analytic continuation that $Y_n(z)$ is also a solution to a corresponding RHP with a jump at $\Gamma_{q^{-1}}$ (recall Definition \ref{scaled curve}). The jump matrix along this curve is given by
\begin{eqnarray*}
\begin{pmatrix}
1 &
w(s)h_\alpha(s) & w(s)h_\beta(s) \\
0 & 1 & 0 \\
0 & 0 & 1 
\end{pmatrix},\, \text{for}\,s\in\Gamma_{q^{-1}} &=& \begin{pmatrix}
1 &
w(q^{-1}s)h_\alpha(q^{-1}s) & w(q^{-1}s)h_\beta(q^{-1}s) \\
0 & 1 & 0 \\
0 & 0 & 1 
\end{pmatrix},\, \text{for}\,s\in\Gamma,\\
&=&\begin{pmatrix}
1 &
(1-s)w(s)q^{-\alpha}h_\alpha(s) & (1-s)w(s)q^{-\beta}h_\beta(s) \\
0 & 1 & 0 \\
0 & 0 & 1 
\end{pmatrix},\, \text{for}\,s\in\Gamma.
\end{eqnarray*}
Define,
\begin{equation}
    T_n(z)  = Y_n(q^{-1}z) \begin{bmatrix}
   (1-z) &
   0 &
   0 \\
   0 &
   q^{\alpha} &
   0\\
   0&0&q^{\beta}
\end{bmatrix}
\end{equation}
by the above arguments it follows that $T_n(z)$ satisfies the conditions
\begin{enumerate}
\begin{subequations}
\item $T_n(z)$ is analytic on $\mathbb{C}\setminus \Gamma$.
\item $T_n(z)$ has continuous boundary values $T_{n,-}(s)$ and $T_{n,+}(s)$ as $z$ approaches $s\in\Gamma$ from $\mathcal D_-$ and $\mathcal D_+$ respectively, where
\begin{gather} \nonumber
T_{n,+}(s) =
T_{n,-}(s)
\begin{pmatrix}
1 &
w(s)h_\alpha(s) & w(s)h_\beta(s) \\
0 & 1 & 0 \\
0 & 0 & 1 
\end{pmatrix}, \; s\in \Gamma .
\end{gather}

\item $T_n(z)$ satisfies
\begin{gather} \nonumber
T_n(z)\begin{pmatrix}
-q^{2n}z^{-2n-1} & 0 & 0\\
0 & q^{-n-\alpha}z^{n} & 0 \\
0 & 0 &q^{-n-\beta}z^{n}
\end{pmatrix}
=
I + \mathcal{O}\left( \frac{1}{z} \right), \text{ as }\ |z| \rightarrow \infty.
\end{gather}
\end{subequations}
\end{enumerate}
Thus, it follows that $D_n(z) = T_n(z)Y_n(z)^{-1}$ has no jump in the complex plane and its entries are entire, with the highest power being $z^1$. In fact, writing $Y_n(z)$ as a power series at $z=\infty$ one finds that the entries of $D_n(z)$ are given by.
\begin{equation}\label{d def}
    D_n(z) = \begin{bmatrix}
    \mu_{1}-zq^{-2n} & \mu_{2} &\mu_{3}\\
   \mu_{4}&q^{n+\alpha}&0 \\
   \mu_{5}&0&q^{n+\beta}
   \end{bmatrix} ,
\end{equation} 
Note that $\mu_{i}$ are constants with respect to $z$ but these constants are functions of $n$.

\section{Understanding the $q$-difference component of the Lax pair}
We have just determined that a solution to RHP I satisfies the $q$-difference equation 
\begin{equation}\label{the main diff}
    Y_n(q^{-1}z) \begin{bmatrix}
   (1-z) &
   0 &
   0 \\
   0 &
   q^{\alpha} &
   0\\
   0&0&q^{\beta}
\end{bmatrix} = D_n(z)Y_n(z).
\end{equation}
However, unfortunately we do not know what the value of $\mu_{i}$ is (see Equation \eqref{d def}) without already knowing the first couple of entries of the series expansion of the solution $Y_n(z)$ at $z=\infty$. In the following lemmas we try to gain some insight into $\mu_{i}$ based on our understanding of RHP I. 
\begin{lemma}
    The eigenvalues of $D_n(0)$ are given by $\lambda = 1,q^{\alpha},q^{\beta}$.
\end{lemma}
\begin{proof}
    The first column of Equation \eqref{the main diff} reads
    \[ Y_n^{(1)}(q^{-1}z)(1-z) = D_n(z)Y_n^{(1)}(z) .\]
    Let $U_n(z) = Y_n^{(1)}(z)(qz;q)_\infty$. It follows that 
    \[ U_n(q^{-1}z) = Y_n^{(1)}(q^{-1}z)(1-z)(qz;q)_\infty,\]
    and subsequently $U_n(z)$ satisfies the $q$-difference equation
    \[ U_n(q^{-1}z) = D_n(z)U_n(z) .\]
    We know that as $Y_n(z)$ solves RHP I, it has entries which are analytic in a neighbourhood about $z=0$, these are the entries which compose $Y_n(z)$ restricted to $\mathcal{D}_-$ (in fact we know that these entries are entire). Thus, as $(qz;q)_\infty$ is entire, the entries of $U_n(z)$ are also entire. Hence, $U_n(z)$ can be written as the converging power series:
\begin{equation}
    U_n(z) =  \sum_{i=0}^\infty\begin{bmatrix}
   a_i z^i\\
   b_i z^i\\
    c_i z^i
\end{bmatrix}.
\end{equation}
Substituting this series representation into the $q$-difference equation satisfied by $U_n(z)$ we find,
\begin{eqnarray*}
   \sum_{i=0}^\infty\begin{bmatrix}
   a_i q^{-i}z^i\\
   b_i q^{-i}z^i\\
    c_i q^{-i}z^i
 \end{bmatrix}&=& \begin{bmatrix}
    \mu_{1}-zq^{-2n} & \mu_{2} &\mu_{3}\\
   \mu_{4}&q^{n+\alpha}&0 \\
   \mu_{5}&0&q^{n+\beta}
   \end{bmatrix}\sum_{i=0}^\infty\begin{bmatrix}
   a_i z^i\\
   b_i z^i\\
    c_i z^i
 \end{bmatrix},\\
 0
 &=& \begin{bmatrix}
    \mu_{1}-q^{-i} & \mu_{2} &\mu_{3}\\
   \mu_{4}&q^{n+\alpha}-q^{-i}&0 \\
   \mu_{5}&0&q^{n+\beta}-q^{-i}
   \end{bmatrix}\sum_{i=0}^\infty\left(\begin{bmatrix}
   a_i z^i\\
   b_i z^i\\
    c_i z^i
 \end{bmatrix}+ \begin{bmatrix}
   a_i q^{-2n}z^{i+1}\\
   0\\
    0
 \end{bmatrix}\right),\\
  0
 &=& \begin{bmatrix}
    \mu_{1}-1 & \mu_{2} &\mu_{3}\\
   \mu_{4}&q^{n+\alpha}-1&0 \\
   \mu_{5}&0&q^{n+\beta}-1
   \end{bmatrix} \begin{bmatrix}
   a_0 \\
   b_0 \\
    c_0 
 \end{bmatrix}+ \mathcal{O}(z).
\end{eqnarray*}
Thus, for the above equation to be consistent we require that 1 is an eigenvalue of $D_n(0)$. Let us now turn our attention to the second column of Equation \eqref{the main diff}. The $q$-difference for this column reads
\[ Y_n^{(2)}(q^{-1}z)q^\alpha = D_n(z)Y_n^{(2)}(z).\]
Repeating our earlier arguments, we conclude that $Y_n^{(2)}(z)$ has entries which are entire and can be written as the converging power series
\begin{equation}
    Y_n^{(2)}(z) =  \sum_{i=0}^\infty\begin{bmatrix}
   a_i z^i\\
   b_i z^i\\
    c_i z^i
\end{bmatrix},
\end{equation}
where clearly $a_i$, $b_i$ and $c_i$ are different from $U_n(z)$. Substituting this series representation into the $q$-difference equation satisfied by $Y_n^{(2)}(z)$ we find,
\begin{eqnarray}
   q^{\alpha}\sum_{i=0}^\infty\begin{bmatrix}
   a_i q^{-i}z^i\\
   b_i q^{-i}z^i\\
    c_i q^{-i}z^i
 \end{bmatrix}&=& \begin{bmatrix}
    \mu_{1}-zq^{-2n} & \mu_{2} &\mu_{3}\\
   \mu_{4}&q^{n+\alpha}&0 \\
   \mu_{5}&0&q^{n+\beta}
   \end{bmatrix}\sum_{i=0}^\infty\begin{bmatrix}
   a_i z^i\\
   b_i z^i\\
    c_i z^i
 \end{bmatrix},\nonumber\\
 0
 &=& \begin{bmatrix}
    \mu_{1}-q^{-i+\alpha} & \mu_{2} &\mu_{3}\\
   \mu_{4}&q^{n+\alpha}-q^{-i+\alpha}&0 \\
   \mu_{5}&0&q^{n+\beta}-q^{-i+\alpha}
   \end{bmatrix}\sum_{i=0}^\infty\left(\begin{bmatrix}
   a_i z^i\\
   b_i z^i\\
    c_i z^i
 \end{bmatrix}+ \begin{bmatrix}
   a_i q^{-2n}z^{i+1}\\
   0\\
    0
 \end{bmatrix}\right),\nonumber\\
  0
 &=& \begin{bmatrix}
    \mu_{1}-q^{\alpha} & \mu_{2} &\mu_{3}\\
   \mu_{4}&q^{n+\alpha}-q^{\alpha}&0 \\
   \mu_{5}&0&q^{n+\beta}-q^{\alpha}
   \end{bmatrix} \begin{bmatrix}
   a_0 \\
   b_0 \\
    c_0 
 \end{bmatrix}+ \mathcal{O}(z).\label{Equation X}
\end{eqnarray}
For the above equation to be consistent we now require that $q^\alpha$ is an eigenvalue of $D_n(0)$. Repeating the same arguments for the third column it follows that $q^{\beta}$ must also be an eigenvalue $D_n(0)$.
\end{proof}

\begin{lemma}\label{coeff determination}
    The entries of $D_n(z)$ satisfy
    \begin{eqnarray}
        \mu_1 &=& 1+q^\alpha +q^\beta - q^{\alpha+n} - q^{\beta+n}.\\
        \mu_2\mu_4 &=& \frac{q^{\alpha-n}(q^n-1)(q^{\alpha+n}-1)(q^{\alpha+n}-q^\beta)}{q^{\beta}-q^\alpha}.\\
        \mu_3\mu_5 &=& \frac{q^{\beta-n}(q^n-1)(q^{\beta+n}-1)(q^{\beta+n}-q^\alpha)}{q^{\alpha}-q^\beta}.
    \end{eqnarray}
\end{lemma}
\begin{proof}
    This follows by direct computation. Calculating det$(D_n(0)-\lambda)=0$ for $\lambda = 1, q^\alpha, q^\beta$ results in three linearly independent equations for $\mu_1$, $\mu_2\mu_4$ and $\mu_3\mu_5$. Solving this system of equations gives the desired solution.
\end{proof}

Having determined the entries of $D_n(z)$ we can use Equation \eqref{the main diff} and Lemma \ref{coeff determination} to deduce a third order $q$-difference equation satisfied by the entries of $Y_n(z)$. Pre-empting our discussion in the next section we will label our independent variable as $t$ instead of $z$. It will be useful to think of $t$ as $t=zq^{-2n}$. That is, $t$ is of order one when $z$ is of order $q^{2n}$. Re-writing the first order, $3\times3$ matrix $q$-difference equation in Equation \eqref{the main diff} as a third order $q$-difference equation for the $(1,1)$ entry: $Y_{1,1}(tq^{2n})(q^{2n+1}t;q)_\infty$, we find
\begin{eqnarray}
y(t) &=& \left[q^n(q^\alpha+q^\beta) + (\mu_1-qt)\right]y(qt) \nonumber \\
&+& \left[(\mu_2\mu_4+\mu_3\mu_5) - q^n\left((q^\alpha+q^\beta)(\mu_1-q^{2}t) +q^{n+\alpha+\beta}\right)\right]y(q^2t) \nonumber \\
&+& \left[ q^{\alpha+\beta}-q^{3+\alpha+\beta+2n}t\right]y(q^3t). \label{1st diff v2}
\end{eqnarray}
Applying Lemma \ref{coeff determination} it follows that $Y_{1,1}(tq^{2n})(q^{2n+1}t;q)_\infty$ satisfies the difference equation
\begin{eqnarray}
y(t) &=& \left[ 1+q^\alpha + q^\beta-qt\right]y(qt) \nonumber \\
&+& \left[-(q^\alpha + q^\beta + q^{\alpha+\beta}) + (q^\alpha+q^\beta)q^{n+2}t\right]y(q^2t) \nonumber \\
&+& \left[ q^{\alpha+\beta}-q^{3+\alpha+\beta+2n}t\right]y(q^3t). \nonumber
\end{eqnarray}
Similarly, one can determine difference equations satisfied by $Y_{2,1}$ and $Y_{3,1}$. These results are collected in the following proposition.
\begin{proposition}
The entries of the first column of the solution $Y_n$ to RHP I, satisfy the following $q$-difference equations:\\
$Y_{1,1}(tq^{2n})(q^{2n+1}t;q)_\infty$ satisfies the difference equation
\begin{eqnarray}
y(t) &=& \left[ 1+q^\alpha + q^\beta-qt\right]y(qt) \nonumber \\
&+& \left[-(q^\alpha + q^\beta + q^{\alpha+\beta}) + (q^\alpha+q^\beta)q^{n+2}t\right]y(q^2t) \nonumber \\
&+& \left[ q^{\alpha+\beta}-q^{3+\alpha+\beta+2n}t\right]y(q^3t).
\label{1st diff v3}
\end{eqnarray}
$Y_{2,1}(tq^{2n})(q^{2n+1}t;q)_\infty$  satisfies the difference equation
\begin{eqnarray}
y(t) &=& \left[ 1+q^\alpha + q^\beta-q^{2}t \right]y(qt) \nonumber \\
&+& \left[-(q^\alpha + q^\beta + q^{\alpha+\beta}) + \left(q^\alpha + q^{\beta +1}\right)q^{n+2}t \right]y(q^2t) \nonumber \\
&+& \left[ q^{\alpha+\beta}-q^{3+\alpha+\beta+2n}t\right]y(q^3t). \label{2nd diff v3}
\end{eqnarray}
$Y_{3,1}(tq^{2n})(q^{2n+1}t;q)_\infty$  satisfies the difference equation
\begin{eqnarray}
y(t) &=& \left[ 1+q^\alpha + q^\beta-q^{2}t \right]y(qt) \nonumber \\
&+& \left[-(q^\alpha + q^\beta + q^{\alpha+\beta}) + \left(q^{\alpha+1} + q^{\beta}\right)q^{n+2}t \right]y(q^2t) \nonumber \\
&+& \left[ q^{\alpha+\beta}-q^{3+\alpha+\beta+2n}t\right]y(q^3t). \label{3rd diff v3}
\end{eqnarray}
\end{proposition}

\section{Determining $P_{n,n}(0)$ in the limit $n\to\infty$}\label{Sect 3}

We will focus the following discussion on the function $Y_{1,1}(tq^{2n})(q^{2n+1}t;q)_\infty = P_{n,n}(tq^{2n})(q^{2n+1}t;q)_\infty$ (recall that $Y_{1,1}(t)$ is the $(1,1)$ entry of $Y_n(t)$). Consider the corresponding $q$-difference equation, Equation \eqref{1st diff v3}. By substituting the power series $y(t) = \sum_{k=0}^\infty c_kt^k$ into Equation \eqref{1st diff v3} one readily finds that there is a unique (up to a scalar factor) solution to Equation \eqref{1st diff v3} which is analytic in a neighbourhood about $t=0$. The coefficients $c_k$ of this solution satisfy the recurrence relation
 \begin{equation}\nonumber
     c_k = \frac{(-q^k)(1-q^{n+\alpha+k})(1-q^{n+\beta+k})}{(1-q^k)(1-q^{k+\alpha})(1-q^{k+\beta})}c_{k-1}.
 \end{equation}
Let $u_n(t)$ be this analytic solution normalised such that $u_n(0)=1$. Thus,
\begin{equation}\label{finding F1}
    u_n(t) = \sum_{k=0}^\infty \frac{(-1)^{k}q^{k(k+1)/2} (q^{n+\alpha+1};q)_k (q^{n+\beta+1};q)_k }{(q;q)_k(q^{\alpha+1};q)_k(q^{\beta+1};q)_k}t^k.
\end{equation}
Note that $u_n(t)$ is an entire function.  Moreover, the power series of $u_n(t)$ immediately implies that for finite $t$, $u_n(t)$ approaches a non-zero limiting function, $u_\infty(t)$ as $n\to\infty$. \\
 
Under our normalisation it is clear that $u_n(t)= P_{n,n}(tq^{2n})(q^{2n+1}t;q)_\infty/P_{n,n}(0)$. Our goal is to find $P_{n,n}(0)$, focusing on its limit as $n\to\infty$. Define $v_n(t) = u_n(t)/g(t)$, where $g(t) = (qt;q)_\infty(t^{-1};q)_\infty$, satisfies the $q$-difference equation 
 \begin{equation}\label{g diff}
  g(qt) = -q^{-1}t^{-1}g(t).
\end{equation}
 By direct computation it follows $v_n(t)$ satisfies the $q$-difference equation
\begin{multline}\label{f mat diff}
   \begin{bmatrix}
   v_n(q^{-3}t)\\
   v_n(q^{-2}t)\\
   v_n(q^{-1}t)
 \end{bmatrix}= \Bigg(\begin{bmatrix}
   1 & 0 &0\\
   1&0&0 \\
   0&1&0
   \end{bmatrix}
   +t^{-1}\begin{bmatrix}
   (1+q^\alpha+q^\beta)q^2 & (q^\alpha+q^\beta)q^{n+2} &0\\
   0&0&0 \\
   0&0&0
   \end{bmatrix}
   \\
   +t^{-2}\begin{bmatrix}
   0 & -q^3(q^\alpha+q^\beta+q^{\alpha+\beta}) &q^{\alpha+\beta+2n+3}\\
   0&0&0 \\
   0&0&0
   \end{bmatrix}
   +t^{-3}\begin{bmatrix}
   0 & 0 &q^{\alpha+\beta+3}\\
   0&0&0 \\
   0&0&0
   \end{bmatrix} \Bigg)
 \begin{bmatrix}
   v_n(q^{-2}t)\\
   v_n(q^{-1}t)\\
   v_n(t)
 \end{bmatrix}.
 \end{multline}
This matrix expression allows us to write the vector $V(q^{-j}t) = [v_n(q^{-2-j}t),v_n(q^{-1-j}t),v_n(q^{-j}t)]^{T}$ as a Pochhammer symbol with matrix elements, which multiplies the vector $V(t)$. That is 
\begin{equation}\label{matrix poch}
    V(q^{-j}t) = \Big( \prod_{i=0}^j\left(A + q^{i-1}t^{-1}B + q^{2i-2}t^{-2}C + q^{3i-2}t^{-3}D\right)\Big)V(t),  
\end{equation}
where the matrices $A,B,C$ and $D$ are given by Equation \eqref{f mat diff}. As 
\[ A^i = \begin{bmatrix}
   1 & 0 &0\\
   1&0&0 \\
   1&0&0
   \end{bmatrix} ,\]
for $i>1$ (note the maximum eigenvalue of $A$ is $\lambda=1$), it follows that the product in Equation \eqref{matrix poch} converges as $j\to\infty$. Consider Equation \eqref{matrix poch} in the case when $t$ is large. In particular, when $t$ is such that the entries of $t^{-1}B$, $t^{-2}C$ and $t^{-3}D$ are much less than $\frac{1}{1-q}$, say by a factor of one hundred. This occurs for all $t$ with modulus greater than some value $t_c$, where such a choice can be made independent of $n$. As will be explained in the remark below, for this choice of $t$ it follows that if $v_n(q^{-2}t)$ is of the same order as $v_n(q^{-1}t)$ and $v_n(t)$, then $v_n(q^{-j}t)$ is of the same order as $v_n(q^{-2}t)$ for all $j>3$. Hence, $v(q^{-j}t)$ approaches a non-zero constant as $j\to\infty$.
\begin{remark}\label{remark 4.1}
  That $v_n(q^{-j}t)$ is of the same order as $v_n(q^{-2}t)$ for all $j>3$ is most easily seen if we first make the simplification that $V(t)$ is a scalar $($let us call this scalar $\nu(t))$. For the sake of simplicity, we take the case of three iterations in $j$, we find
\begin{eqnarray*}
    \nu(q^{-3}t) &=& (A+\epsilon)(A+q\epsilon) (A+q^2\epsilon)\nu(t) ,\\
    &=& \Big(A^3 + \epsilon A^2(1+q+q^2) + A\epsilon^2q(1+q+q^2) + q^3\epsilon^3\Big)\nu(t).
\end{eqnarray*}
If $A=1$ and $\epsilon$ is much less then $\frac{1}{1-q}$, say by a factor of one hundred, then it follows $\nu(q^{-3}t)$ is approximately $\nu(t)$. One can easily extend this to any $j>3$. The matrix case follows similarly, but the notation is more cluttered due to the lack of commutativity. Crucially, in the matrix case the elements of $V(t)$ can be of different order i.e. $\mathcal{O}(\epsilon)$. If all the elements of $V(t)$ are of the same order then the argument proceeds exactly as in the scalar case. 
\end{remark}

As we have established, if $v_n(q^{-2}t)$ is of the same order as $v_n(q^{-1}t)$ and $v_n(t)$ for $|t|>t_c$, then it is true that $v_n(q^{-j}t)$ is of the same order as $v_n(t)$ as $j\to\infty$. This observation allows us to deduce a number of facts about $v_n(t)$. \\

Recall that $u_n(t) = P_{n,n}(tq^{2n})w(tq^{2n})/P_{n,n}(0)$ and $v_n(t)=u_n(t)/g(t)$. Using induction on Equation \eqref{g diff} we find
\[ g(q^{-2n}t) = t^{2n}q^{-n(2n-1)}g(t) ,\]
which in turn allows us to determine 
\begin{eqnarray}
    v_n(q^{-2n}t) &=& \frac{u_n(q^{-2n}t)}{g(q^{-2n}t)}\nonumber \\
    &=& \frac{P_{n,n}(t)w(t)}{t^{2n}q^{-n(2n-1)}g(t)P_{n,n}(0)}\nonumber \\
    &=& q^{n(2n-1)}\frac{P_{n,n}(t)}{t^{2n}(t^{-1};q)_\infty P_{n,n}(0)}\label{eq vnqn}
\end{eqnarray}
Thus, $v_n(t)$ approaches a non-zero constant as $t\to\infty$. Hence, it is required that the poles of $v_n(t) = u_n(t)/g(t)$ vanish as $t\to\infty$. The poles of $u_n(t)/g(t)$ occur at the zeros of $g(t)$ which are at $t = q^{k}$ for $k\in\mathbb{Z}$. Hence, we require that near such points
\begin{equation}\label{vanishing f} v_n(q^{-2}q^{-k}) = \mathcal{O}\Big(\max\big(q^{k}q^{n}v_n(q^{-1}q^{-k}), q^{2k}v_n(q^{-1}q^{-k}), q^{2k}q^{2n}v_n(q^{-k}),q^{3k}v_n(q^{-k})\big)\Big) ,
\end{equation}
for all $k$ such that $q^{-k}>t_c$ (recall $k\in \mathbb{N}$). If this were not the case then the $v_n(q^{-2}q^{-k})$ term would remain dominant and by Equation \eqref{matrix poch} the pole would persist for all $q^{-k}$ as $k\to\infty$. 
\begin{remark}\label{un 0s}
Equation \eqref{vanishing f} and the identity $u_n(t) = v_n(t)g(t)$ tell us that
    \begin{equation}\nonumber u_n(q^{-2}q^{-k}) = \mathcal{O}\Big(\max\big(q^{n}u_n(q^{-1}q^{-k}), q^{k}u_n(q^{-1}q^{-k}), q^{2n}u_n(q^{-k}),q^{k}u_n(q^{-k})\big)\Big) ,
\end{equation}
for all $k$ such that $q^{-k}>t_c$, which in turn tells us that $P_{n,n}(q^{2n-k})$ is rapidly vanishing as $k$ increases.  
\end{remark}

We have established that $v_n(t)$ approaches a non-zero constant as $t\to\infty$. Let us call this constant $\Omega_n^{(1)}$. It follows that
\[ v_n(t) = \Omega_n^{(1)} + \sum_{j=-2n}^{\infty} \frac{\text{Res}\left( \frac{u_n(q^j)}{g(q^j)} \right)}{t-q^{j}} ,\]
where, as we detailed in Equation \eqref{vanishing f} and Remark \ref{un 0s}, the residue of $\frac{u_n(q^j)}{g(q^j)} = v_n(q^j)$ is  rapidly vanishing as $j \to-\infty$. Note that
\begin{equation}\nonumber
    \Omega_n^{(1)} = \lim_{|t|\to\infty} v_{n}(t) = \lim_{|t|\to\infty} \frac{u_n(t)}{g(t)} ,
\end{equation}
where we take this limit away from $t=q^k$ for $k\in\mathbb{Z}$. By the arguments presented in Remark \ref{remark 4.1} we know that this limit exists and is non-zero. The constant $\Omega_n^{(1)}$ can readily be calculated to as high a degree of accuracy as desired using Equation \eqref{f mat diff} and the power series representation of $u_n(t)$. Furthermore, $u_n(t)$ approaches a limiting function $u_\infty(t)$. This function is an entire function and hence increases along a ray as $t$ increases. It follows that along this ray $u_\infty(t)$ (and $u_n(t)$ for large enough $n$) violates the condition in Remark \ref{un 0s}. Hence, $\Omega_n^{(1)}$ approaches the non-zero constant, 
\begin{equation}\label{omega 1 def}
    \Omega^{(1)}_\infty = \lim_{|t|\to\infty}\frac{u_\infty(t)}{g(t)},
\end{equation}
as $n\to\infty$. See Figure \ref{u convergence} for an illustration of how to determine $\Omega^{(1)}_\infty$. Applying Equation \eqref{eq vnqn} and \cite[Thereom 4]{postelmans2005multiple} we find
\begin{equation}\label{final Pn0}
    \lim_{n\to\infty} q^{-n(2n-1)}P_{n,n}(0) = \frac{1}{\Omega^{(1)}_\infty}.
\end{equation}
\begin{remark}
    Equation \eqref{final Pn0} tells us the magnitude of $P_{n,n}(0)$ as $n\to\infty$. However, as $P_{n,n}(q^{2n}t)(q^{2n-1}t;q)_\infty$ is a solution to Equation \eqref{1st diff v3} we can conclude that for any finite $t$ \[\lim_{n\to\infty} q^{-n(2n-1)}P_{n,n}(q^{2n}t) = F_1(t) ,\] 
    where $F_1(t)$ is an entire function of $t$ which is independent of $n$. Furthermore, looking at the limit $n\to\infty$ of Equation \eqref{finding F1} we determine that
    \[ F_1(t) = \frac{1}{\Omega^{(1)}_\infty}\sum_{j=0}^\infty \frac{(-1)^jq^{j(j+1)/2}}{(q^{\alpha+1};q)_j(q^{\beta+1};q)_j(q;q)_j}t^j.\]
\end{remark}
The results of this section are summarised in the lemma below. Although we only discussed the behaviour of $P_{n,n}(q^{2n}t)$ in this section, repeating the presented arguments yields similar insights for $P_{n,n-1}(q^{2n}t)$ and $P_{n-1,n}(q^{2n}t)$, the $(2,1)$ and $(3,1)$ entries of $Y_n$.

\begin{lemma}\label{Section 3 lemma}
    The monic multiple $q$-orthogonal polynomials $P_{n,n}(z)$, $P_{n-1,n}(z)$ and $P_{n,n-1}(z)$ defined in Equation \eqref{ortho 1} display the following behaviours.
    \begin{enumerate}
    \item $\lim_{n\to\infty} q^{n(1-2n)}P_{n,n}(q^{2n}z) =  F_1(z)$, where 
     \[F_1(z)=\frac{1}{\Omega^{(1)}_\infty}\sum_{j=0}^\infty \frac{(-1)^jq^{j(j+1)/2}}{(q^{\alpha+1};q)_j(q^{\beta+1};q)_j(q;q)_j}z^j,\]
     and $\Omega^{(1)}_\infty$ is a non-zero constant defined in Equation \eqref{omega 1 def}.
        \item For $k\in\mathbb{N}$ larger than some critical value $k_c$ (which is independent of $n$) it holds that
        \begin{multline}\nonumber
            P_{n,n}(q^{2n-2}q^{-k})(q^{2n+1}q^{-k};q)_\infty = \\ \mathcal{O}\Big(\max\big(q^{n}P_{n,n}(q^{-1}q^{2n-k}), q^{k}P_{n,n}(q^{-1}q^{2n-k}), q^{2n}P_{n,n}(q^{2n-k}),q^{k}P_{n,n}(q^{2n-k})\big)\Big).
        \end{multline}
        Note that for $k<2n+1$, $(q;q)_\infty<(q^{2n+1}q^{-k};q)_\infty<1$.
        \item $\lim_{n\to\infty} q^{n(3-2n)}P_{n,n-1}(q^{2n}z) = F_2(z)$, where $F_2(z)=\lambda_1F_1(qz)$, and $\lambda_1$ is a non-zero constant.
        \item For $k\in\mathbb{N}$ larger than some critical value $k_c$ (which is independent of $n$) it holds that
        \begin{multline}\nonumber 
            P_{n,n-1}(q^{2n-2}q^{-k})(q^{2n+1}q^{-k};q)_\infty = \\ \mathcal{O}\Big(\max\big(q^{n}P_{n,n-1}(q^{-1}q^{2n-k}), q^{k}P_{n,n-1}(q^{-1}q^{2n-k}), q^{2n}P_{n,n-1}(q^{2n-k}),q^{k}P_{n,n}(q^{2n-k})\big)\Big).
        \end{multline}
        \item $\lim_{n\to\infty} q^{n(3-2n)}P_{n-1,n}(q^{2n}z) = \lambda_3 F_2(z)$. Note that up to linear scaling both $P_{n,n-1}(zq^{2n})$ and $P_{n-1,n}(zq^{2n})$ approach the same function $F_2(z)$.
        \item For $k\in\mathbb{N}$ larger than some critical value $k_c$ (which is independent of $n$) it holds that
        \begin{multline}\nonumber
            P_{n-1,n}(q^{2n-2}q^{-k})(q^{2n+1}q^{-k};q)_\infty =  \\\mathcal{O}\Big(\max\big(q^{n}P_{n,n-1}(q^{-1}q^{2n-k}), q^{k}P_{n,n-1}(q^{-1}q^{2n-k}), q^{2n}P_{n,n-1}(q^{2n-k}),q^{k}P_{n,n}(q^{2n-k})\big)\Big).
        \end{multline}
    \end{enumerate}
\end{lemma}

\begin{figure}[h]
    \centering
    \includegraphics[width=11cm]{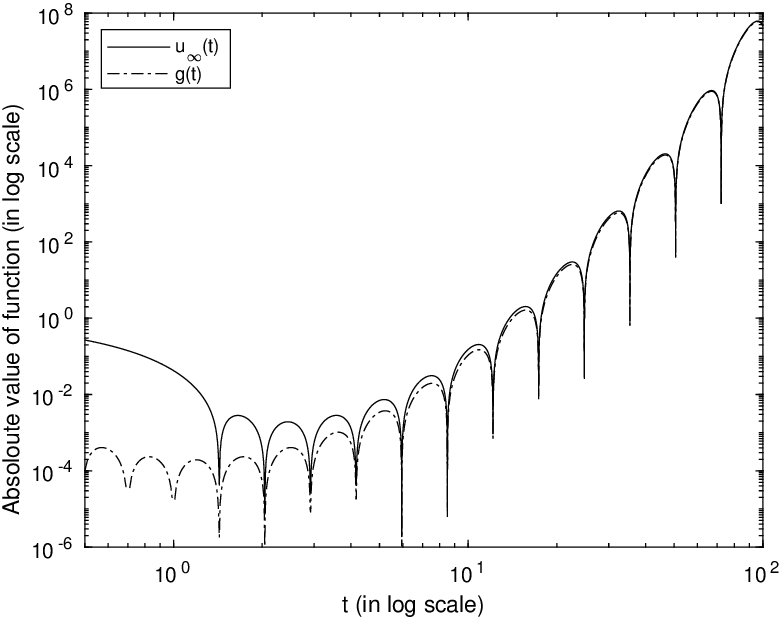}
    \caption{Illustrative example of $u_\infty(t)$ converging to $g(t)$ with the parameters $q=0.7$, $\alpha = 0.4$ and $\beta=0.6$. Indeed, in this case we find that
    $ \Omega_\infty^{(1)} = \lim_{t\to\infty}\frac{u_\infty(t)}{g(t)} = 1.$ Note that the zeros of $u_\infty(t)$ are slightly different to the zeros of $g(t)$, which is difficult to see in the figure. As $k$ becomes large the zeros of $u_{\infty}(t)$ approach the zeros of $g(t)$ which occur at $q^{-k}$ for $k\in\mathbb{Z}$. Indeed, Lemma \ref{Section 3 lemma} shows that $u_{\infty}(q^{-k})$ becomes vanishingly small as $k\to\infty$.}
    \label{u convergence}
\end{figure}

\section{Determining $\gamma_1^{(n,n)}$ in the limit $n\to\infty$}\label{Sect 4}
We will now use a similar approach to Section \ref{Sect 3} to determine the magnitude of $\gamma_1^{(n,n)}$ (defined in Equation \eqref{gamma def}) for multiple little $q$-Jacobi polynomials of the first kind as $n\to\infty$. We will study the function
\[ y_n(z) = \frac{1}{\gamma_1^{(n,n)}}\sum_{k=0}^\infty \frac{P_{n,n}(q^k)w_1(q^k)q^k}{z-q^k},\]
at $y_n(q^{2n}z)$ as $n\to\infty$. This will in turn allow us to deduce the magnitude of $\gamma_1^{(n,n)}$ as $n\to\infty$. Note that one has to avoid the points $z=q^k$ for $k\in\mathbb{Z}_{\geq 0}$, as $y_n(z)$ has singularities at these locations. \\

Applying Equation \eqref{the main diff} and \eqref{1st diff v3} we find that $y_n(z)$ satisfies the third order $q$-difference equation
\begin{eqnarray}
q^{3\alpha}y_n(z) &=& q^{2\alpha}\left[ 1+q^\alpha + q^\beta-q^{1-2n}z\right]y_n(qz) \nonumber \\
&+& q^{\alpha}\left[-(q^\alpha + q^\beta + q^{\alpha+\beta}) + (q^\alpha+q^\beta)q^{2-n}z\right]y_n(q^2z) \nonumber \\
&+& \left[ q^{\alpha+\beta}-q^{3+\alpha+\beta}z\right]y_n(q^3z), 
\end{eqnarray}
By the orthogonality condition, Equation \eqref{ortho 1}, it follows that $\lim_{z\to\infty} y_n(z)z^{n} = z^{-1}$. Defining $u_n(z) = y_n(z)z^{n}$ we deduce that $u_n(z)$ satisfies the $q$-difference equation
\begin{eqnarray}
q^{2\alpha-\beta}q^{3n}u_n(z) &=& q^{\alpha-\beta}\left[ q^{2n}(1+q^{\alpha} + q^{\beta})-qz\right]u_n(qz) \nonumber \\
&+& q^{\alpha}\left[-q^{n}(q^{-\alpha} + q^{-\beta} + 1) + (q^{-\alpha}+q^{-\beta})q^{2}z\right]u_n(q^2z) \nonumber \\
&+& \left[ 1-q^{3}z\right]u_n(q^3z), \label{vn diff}
\end{eqnarray}
There is only one solution to Equation \eqref{vn diff} which can be written as a power series at $z=\infty$ with the required asymptotic behaviour $\lim_{z\to\infty} \varphi(z) = z^{-1}$, let us denote this solution as $\varphi(z)$. It follows that $\varphi(z) = u_n(z) = y_n(z)z^{n}$. Furthermore, writing $\varphi(z)$ ($=u_n(z)$) as a power series at infinity and studying the corresponding coefficients we determine that $u_n(z)$ approaches a limiting function $u_\infty(z)$ as $n\to\infty$ for $|z|>1$. \\

We shall now investigate the behaviour of $u_n(zq^{2n})$ for large $n$, which will in turn tell us about $y_n(q^{2n}z)$. First, let us re-write Equation \eqref{vn diff} as,
\begin{eqnarray}\label{vn diff v2}
u_n(z) &=& q^{\alpha}\frac{q^{n}(q^{-\alpha} + q^{-\beta} + 1) - (q^{-\alpha}+q^{-\beta})q^{-1}z}{1-z}u_n(q^{-1}z) \nonumber \\
&-& q^{\alpha-\beta}\frac{ q^{2n}(1+q^{\alpha} + q^{\beta})-q^{-2}z}{1-z}u_n(q^{-2}z) \nonumber \\
&+& \frac{q^{2\alpha-\beta}q^{3n}}{1-z}u_n(q^{-3}z). 
\end{eqnarray}
Define the function $v_n(z) = u_n(z)/f(z)$, where $f(z)$ satisfies the $q$-difference equation 
\begin{equation}\label{leading diff}
    f(q^{-1}z) = q^{\frac{\beta-\alpha}{2}}q^{\frac{3}{4}}z^{-\frac{1}{2}}f(z).
\end{equation}
We are choosing this function $f(z)$ as it will strip away the leading order behaviour of $u_n(z)$ as $z\to0$. Using Equations \eqref{vn diff v2} and \eqref{leading diff} we find that $v_n(z)$ satisfies the difference equation
\begin{eqnarray}\label{wn diff}
v_n(z) &=& \left(q^{\frac{\beta+\alpha}{2}+\frac{3}{4}}\right)\frac{(q^{2n}z^{-1})^{\frac{1}{2}}(q^{-\alpha} + q^{-\beta} + 1) - (q^{-\alpha}+q^{-\beta})q^{-1}z^{\frac{1}{2}}}{1-z}v_n(q^{-1}z) \nonumber \\
&-& \frac{ \left(q^{2n}z^{-1}\right)q^2(1+q^{\alpha} + q^{\beta})-1}{1-z}v_n(q^{-2}z) \nonumber \\
&+& \frac{\left(q^{2n}z^{-1}\right)^{\frac{3}{2}}q^{\frac{\alpha+\beta}{2}}q^{\frac{17}{4}} }{1-z}v_n(q^{-3}z). \nonumber 
\end{eqnarray}
We will now apply the same trick as was used in Section \ref{Sect 3} to deduce $P_{n,n}(0)$. This time however, we will be approaching $z\sim q^{2n}$ from $z\sim 1$ and not the other way around. Re-writing Equation \eqref{wn diff} as a vector $q$-difference equation we find,
\begin{multline}\label{wn matrix equation}
   \begin{bmatrix}
   v_n(z)\\
   v_n(q^{-1}z)\\
   v_n(q^{-2}z)
 \end{bmatrix}= \Bigg(\begin{bmatrix}
   0 & \frac{1}{1-z} &0\\
   1&0&0 \\
   0&1&0
   \end{bmatrix}
   -\frac{z^{1/2}q^{(\frac{\beta+\alpha}{2}-\frac{1}{4}) }}{1-z} \begin{bmatrix}
   (q^{-\alpha}+q^{-\beta}) & 0 & 0\\
   0&0&0 \\
   0&0&0
   \end{bmatrix}  \\
   +\frac{\left(q^{2n}z^{-1}\right)}{1-z}^{\frac{1}{2}}\begin{bmatrix}
   q^{\frac{\alpha+\beta}{2}+\frac{3}{4}}(1+q^{-\alpha}+q^{-\beta})q^2 & -q^2(1+q^{\alpha} + q^{\beta})\left(q^{2n}z^{-1}\right)^{\frac{1}{2}} &q^{\frac{\alpha+\beta}{2}+\frac{17}{4}}\left(q^{2n}z^{-1}\right)\\
   0&0&0 \\
   0&0&0
   \end{bmatrix}
 \Bigg)
 \begin{bmatrix}
   v_n(q^{-1}z)\\
   v_n(q^{-2}z)\\
   v_n(q^{-3}z)
 \end{bmatrix}.
 \end{multline}
Using Equation \eqref{wn matrix equation} we can write $V_n(q^jz) = [v_n(q^{j}z), v_n(q^{j-1}z),v_n(q^{j-2}z)]^T$ in terms of \\ $V_n(z) = [v_n(q^{-1}z), v_n(q^{-2}z),v_n(q^{-3}z)]^T$ for any $j$. Applying the same arguments presented in the proof of Lemma \ref{Section 3 lemma} it follows that $v_n(q^{2n}z) = \Lambda_{n}(z)v_n(z)$ for some function $\Lambda_{n}(z)$, where the entries of $\Lambda_n(z)$ are $\mathcal{O}(1)$. One has to be careful of two potential issues however. For $z>q^n$, the second most dominant term is $\mathcal{O}(z^{1/2})$ (after $\mathcal{O}(1)$ in the first matrix). In this case only the first and second entry of the vector contribute to its change. One can readily verify that in this reduced case the eigenvalues, $\lambda_{1,2}$, are bounded above and below by $\lambda_{1,2} = 1\pm kz^{1/2}$ for some bounded constant $k$ independent of $n$. As a result we do not have to be concerned with the potential that $v(q^iz)\ll v(z)$. For $z<q^n$ the second most dominant term is $\mathcal{O}(q^nz^{-1/2})$. Thus, when calculating $v(q^{2n}z)$ one should be calculate one product from some $z_{c,1}$ near $z$, and one from another $z_{c,2}$ near $zq^{2n}$, and there will be a vanishing error away from these points. We conclude that $\Lambda_{n}(z)$ approaches a non-zero limiting function $\Lambda_{\infty}(z)$ as $n\to\infty$. Furthermore, by considering the power series representation of $v_n(z)$ we can deduce that $v_n(z)$ also approaches a limiting function $v_\infty(z)$ as $z\to\infty$. \\  

Thus, we find that
\begin{eqnarray}
    v_n(q^{2n}z) &=& \Lambda_{n}(z)v_n(z)\nonumber \\
    \frac{u_n(q^{2n}z)}{f(q^{2n}z)} &=& \Lambda_{n}(z)\frac{u_n(z)}{f(z)} \nonumber \\
    \frac{y_n(q^{2n}z)\left( q^{2n^2}z^n\right)}{f(q^{2n}z)} &=& \Lambda_{n}(z)\frac{ u_n(z)}{f(z)} \nonumber \\
    \frac{y_n(q^{2n}z)\left( q^{2n^2}z^n\right)f(z)}{f(q^{2n}z)} &=& \Lambda_{n}(z)u_n(z)\nonumber \\
    u_n(z) &=& \frac{y_n(q^{2n}z)\left( q^{n(n+1+\beta-\alpha)}\right)}{\Lambda_{n}(z)} \label{step 2}
\end{eqnarray}
Recall from Lemma \ref{Section 3 lemma} that
\[ \lim_{n\to\infty} q^{-n(2n-1)}P_{n,n}(q^{2n}z) = F_1(z),\]
and furthermore, $P_{n,n}(q^{2n-k})$ is rapidly vanishing as $k\to\infty$. Hence, it follows that 
\begin{equation}\label{step 1}
\lim_{n\to\infty}y_n(q^{2n}z)\gamma_1^{(n,n)}q^{-n(2n-1)}q^{-2n\alpha} = \sum_{k= -\infty}^\infty \frac{F_1(z)q^{k(1+\alpha)}}{z-q^k}:=G_1(z).
\end{equation}
Combining Equations \eqref{step 2} and \eqref{step 1} we find
\begin{eqnarray}\label{C1 def}
    \lim_{n\to\infty} q^{-n(3n+\beta+\alpha)}u_n(z)\gamma_1^{(n,n)} &=& \frac{G_1(z)}{\Lambda_{\infty}(z)} .\nonumber \\
    \lim_{n\to\infty} q^{-n(3n+\beta+\alpha)}\gamma_1^{(n,n)} &=& \frac{G_1(z)}{\Lambda_{\infty}(z)u_\infty(z)} .
\end{eqnarray} 

The results of this section are summarised in the lemma below. Although we only discussed the behaviour of $\gamma_1^{(n,n)}$ in this section, repeating the presented arguments yields similar insights for $\gamma_1^{(n-1,n)}$ and $\gamma_2^{(n,n-1)}$.

\begin{lemma}\label{Section 4 lemma}
    The norms, $\gamma_1^{(n,n)}$, $\gamma_1^{(n-1,n)}$ and $\gamma_2^{(n,n-1)}$, defined in Equation \eqref{gamma def} display the following asymptotic behaviour as $n\to\infty$.
    \begin{enumerate}
    \item $\lim_{n\to\infty}q^{-n(3n+\alpha+\beta)}\gamma_1^{(n,n)} = C_1$, where $C_1$ is a non-zero constant which can be readily deduced from Equation \eqref{C1 def}.
    \item $\lim_{n\to\infty}q^{-n(3n+\alpha+\beta-3)}\gamma_1^{(n-1,n)} = C_2$, where $C_2$ is a non-zero constant which can be readily deduced using similar arguments to those that arrived at Equation \eqref{C1 def}.
    \item $\lim_{n\to\infty}q^{-n(3n+\alpha+\beta-3)}\gamma_2^{(n,n-1)} = C_3$, where $C_3$ is a non-zero constant which can also be readily deduced using similar arguments to those that arrived at Equation \eqref{C1 def}.
    \end{enumerate}
\end{lemma}

\begin{remark}
    To give the reader a reference for the magnitude of $\Lambda_n(z)$ and its rate of convergence, we calculated the matrix $M_n(z) = \prod_{j=0}^{2n} T_n(q^jz)$, where $T_n(z)$ is given by the relation in Equation \eqref{wn matrix equation}:
\begin{equation}\nonumber 
   \begin{bmatrix}
   v_n(z)\\
   v_n(q^{-1}z)\\
   v_n(q^{-2}z)
 \end{bmatrix}= T_n(z)
 \begin{bmatrix}
   v_n(q^{-1}z)\\
   v_n(q^{-2}z)\\
   v_n(q^{-3}z)
 \end{bmatrix},
 \end{equation}
 for different values of $n$ at $z=1/2$, and $q=0.7$, $\alpha= 0.4$, $\beta = 0.6.$
\end{remark}

\begin{eqnarray*}
    M_{n=10}(1/2) &=& \begin{bmatrix}
   0.7880&0.4433&\mathcal{O}(q^{3n})\\
   0.9234&-0.4799&\mathcal{O}(q^{3n})\\
   -0.0669&0.0881&\mathcal{O}(q^{3n})
 \end{bmatrix},\\
 M_{n=20}(1/2) &=& \begin{bmatrix}
   0.8770&-0.4796&\mathcal{O}(q^{3n})\\
   0.8127&-0.4039&\mathcal{O}(q^{3n})\\
  -0.1871&0.1579&\mathcal{O}(q^{3n})
 \end{bmatrix},\\
 M_{n=40}(1/2) &=& \begin{bmatrix}
   0.8790&-0.4803&\mathcal{O}(q^{3n})\\
   0.8094&-0.40316&\mathcal{O}(q^{3n})\\
  -0.1903&0.1598&\mathcal{O}(q^{3n})
 \end{bmatrix}.
\end{eqnarray*} 

\section{Universal results}\label{Uni section}

Having comprehensively described the asymptotics of multiple little $q$-Jacobi polynomials with weights $w_1(x) = x^\alpha (qx;q)_\infty$ and $w_2(x) = x^\beta (qx;q)_\infty$ we look to extend our results to a wider class of $q$-orthogonal polynomials. In particular, we will study the asymptotic behaviour of multiple $q$-orthogonal polynomials which are orthogonal with respect to weights of the form
\begin{equation}\label{comparable weights}
    w_3(x) = x^\alpha\omega(x),\qquad w_4(x) = x^\beta\omega(x),
\end{equation}
where the function $\omega(x)$ satisfies the constraint given by Equation \eqref{the weight constraint}. We will show that similar to the case of $q$-orthogonal polynomials with only one measure, studied in \cite{NJTLconst}, we obtain universal asymptotic results independent of the function $\omega(x)$. In order to achieve this goal we will make a series of transformations to RHP I. First we deform the jump contour $\Gamma$ to $\Gamma_{q^{2n}}$ and write an analogues RHP (RHP II) but with an accompanying residue condition. The solution to RHP II is equivalent to the solution of RHP I.
A $3\times3$ matrix function $Y_n(z)$ is a solution to RHP II if it satisfies the following conditions:
\begin{enumerate}[label={{\rm (\roman *)}}]
\begin{subequations}
\item $Y_n(z)$ is analytic on $\mathbb{C}\setminus \left( \Gamma_{q^{2n}}\cup \{ q^{k}\}_{k=0}^{2n-1}\right)$.

\item $Y_n(z)$ has continuous boundary values $Y_{n,-}(s)$ and $Y_{n,+}(s)$ as $z$ approaches $s\in\Gamma_{q^{2n}}$ from $\mathcal D_-$ and $\mathcal D_+$ respectively, where
\begin{gather} 
Y_{n,+}(s) =
Y_{n,-}(s)
\begin{pmatrix}
1 &
w(s)h_\alpha(s) & w(s)h_\beta(s) \\
0 & 1 & 0 \\
0 & 0 & 1 
\end{pmatrix}, \; s\in \Gamma_{q^{2n}} ,
\end{gather}
where $w(s) = (sq;q)_\infty$ and $h_\alpha(s)$ and $h_\beta(s)$ are as discussed in Section \ref{Prelim sect}.

\item $Y_{n}(z)$ satisfies the residue condition
\begin{gather} 
\mathrm{Res}(Y_{n}(q^k))
 =
  \lim_{z\to q^k} Y_{n}(z)
  \begin{bmatrix}
   0 &
   w_1(q^k)q^k &
   w_2(q^k)q^k \\
   0 &
   0 &
   0\\
   0&0&0
   \end{bmatrix},
\end{gather}
for $0\leq k \leq 2n-1$, where $w_1(x)$ and $w_2(x)$ are the weights corresponding to multiple little $q$-Jacobi polynomials of the first kind (see Equation \eqref{the weights}). 

\item $Y_n(z)$ satisfies the asymptotic condition
\begin{gather} 
Y_n(z)\begin{pmatrix}
z^{-2n} & 0 & 0\\
0 & z^{n} & 0 \\
0 & 0 &z^{n}
\end{pmatrix}
=
I + \mathcal{O}\left( \frac{1}{z} \right), \text{ as }\ |z| \rightarrow \infty.
\end{gather}

\end{subequations}
\end{enumerate}

We will now transform RHP II to a form more conducive to studying the asymptotic behaviour of $Y_n(z)$. Define the functions
\begin{eqnarray*}
 a(z) &=& (z^{-1};q)_\infty ,\\
 b(z) &=& (z^{-1};q^2)_\infty ,\\
 c(z) &=& (z^{-1}q;q^2)_\infty .
\end{eqnarray*}
Note that $a(z)=b(z)c(z)$. One can show by direct calculation that they satisfy the $q$-difference equations
\begin{eqnarray*}
 a(qz) &=& (1-z^{-1}q^{-1})a(z) ,\\
 b(qz) &=& (1-z^{-1}q^{-1})c(z) ,\\
 c(qz) &=& b(z).
\end{eqnarray*}
By induction the above difference equations give us:
\begin{subequations}
   \begin{alignat}{2}\label{induct diff}
    a(q^{2n}z) &= q^{-n(2n+1)}z^{-2n}(qz;q)_{2n}a(z)&&:= q^{-n(2n+1)}z^{-2n}A_n(z),\\
    b(q^{2n}z) &= (-1)^nq^{-n(n+1)}z^{-n}(zq^2;q^2)_{n-1}b(z) &&:= (-1)^nq^{-n(n+1)}z^{-n} B_n(z),\\
    c(q^{2n}z) &= (-1)^nq^{-n^2}z^{-n}(zq;q^2)_{n-1}c(z)&&:= (-1)^nq^{-n^2}z^{-n}C_n(z).
\end{alignat} 
\end{subequations}

We will make the transformation
\begin{gather} \label{U transform}
\tilde{U}_n(z)=
\begin{cases}
Y_{n}(z)
\begin{bmatrix}
z^{-2n}a(z)^{-1} & 0 &0 \\
0 & z^{n+\frac{1}{2}}b(z) &0\\
0&0&z^nc(z) 
\end{bmatrix},\qquad \text{for}\,z\in \text{ext}(\Gamma_{q^{2n}}),\\
Y_{n}(z), \qquad  \text{for}\,z \in \text{int}(\Gamma_{q^{2n}}).
\end{cases}
\end{gather}
After this transformation we arrive at a corresponding RHP for $\tilde{U}_n(z)$. However, we will in fact describe the RHP for $U_n(z) = \tilde{U}_n(q^{2n}z)$. To remind the reader that we are making this change of variable we will write the RHP in terms of $t$ as the complex independent variable instead of $z$. We will also make the change
\begin{subequations}\label{weight tilde scaling}
    \begin{eqnarray}
        \tilde{w}_1(z) &=& q^{-2n\alpha}w_1(z),\\
        \tilde{w}_2(z) &=& q^{-2n\alpha}w_2(z).
    \end{eqnarray}
\end{subequations}

$U_n(t)$ is a solution to RHP III if it satisfies the following conditions
\begin{enumerate}[label={{\rm (\roman *)}}]
\begin{subequations}
\item $U_n(t)$ is analytic on $\mathbb{C}\setminus \left( \Gamma_{q}\cup (-\infty,-I)\cup \{ q^{-k}\}_{k=1}^{2n}\right)$, where $-I$ is the minimum point where $\Gamma$ intersects the negative real axis.

\item $U_n(t)$ has continuous boundary values $U_{n,-}(s)$ and $U_{n,+}(s)$ as $t$ approaches $s\in\Gamma$ from $\mathcal D_-$ and $\mathcal D_+$ respectively, and satisfies the jump condition $U_{n,+}(s) = U_{n,-}(s)J_n(s)$ where
\begin{gather} \label{Ut transform}
J_n(s) = 
\begin{bmatrix}
q^{n(1-2n)}A_n(s)^{-1} & q^{n^2}(-1)^ns^\frac{1}{2}B_n(s)w(sq^{2n})h^\alpha(s) &q^{n^2}(-1)^nC_n(s)w(sq^{2n})h^\beta(s) \\
0 & q^{n^2}(-1)^ns^\frac{1}{2}B_n(s) &0\\
0&0&q^{n^2}(-1)^nC_n(s)
\end{bmatrix},
\end{gather}
$w(s) = (sq;q)_\infty$ and $h_\alpha(s)$ and $h_\beta(s)$ are as discussed in Section \ref{Prelim sect}.

\item $U_n(t)$ has continuous boundary values $U_{n,-}(s)$ and $U_{n,+}(s)$ as $t$ approaches $s\in(-\infty,-I)$ from $x+i\epsilon$ and $x-i\epsilon$ respectively, and satisfies the jump condition $U_{n,+}(s) = U_{n,-}(s)M_n(s)$ where
\begin{gather} 
M_n(s) = 
\begin{bmatrix}
1&0&0 \\
0 & -1 &0\\
0&0&1
\end{bmatrix}
\end{gather}
(note that this jump condition is necessary to remedy the multiplication of the second column by $t^{1/2}$, it is not of importance to the rest of our argument).

\item $U_{n}(t)$ satisfies the residue condition
\begin{equation} 
 \mathrm{Res}_{q^{-2k},\,1\leq k\leq n}U(t)
 = 
  \lim_{t\to q^{-2k}}(t- q^{-2k}) U(t)
  \begin{bmatrix}
   0 &
   0 &
   0 \\
   \frac{q^{n(n-1)}}{t^\frac{1}{2}B_n(t)A_n(t)w(tq^{2n})h^\alpha(t)} &
   0 &
   \frac{C_n(t)h^\beta(t)}{t^\frac{1}{2}B_n(t)h^\alpha(t)}\\
   0& 0 &0
   \end{bmatrix},
\end{equation}
\begin{equation}
\mathrm{Res}_{q^{-2k-1},\,0\leq k<n}U(t) = 
 \lim_{t\to q^{-2k-1}}(t- q^{-2k-1}) U(t)
  \begin{bmatrix}
   0 &
   0 &
   0 \\
   0 &
   0 &
   0\\
   \frac{q^{n(n-1)}}{C_n(t)A_n(t))w(tq^{2n})h^\beta(t)} & \frac{t^\frac{1}{2}B_n(t)h^\alpha(t)}{C_n(t)h_\beta(t)} &0
   \end{bmatrix}.
\end{equation}

\item $U_n(t)$ satisfies the asymptotic condition
\begin{gather} 
U_n(z)\begin{pmatrix}
1 & 0 & 0\\
0 & q^{-n}t^{-\frac{1}{2}} & 0 \\
0 & 0 &1
\end{pmatrix}
=
I + \mathcal{O}\left( \frac{1}{t} \right), \text{ as }\ |t| \rightarrow \infty.
\end{gather}
    
\end{subequations}
\end{enumerate}

We need to make one more transformation before we start comparing solutions with different weights to the original multiple little $q$-Jacobi polynomials studied in Sections \ref{Sect 3} and \ref{Sect 4}. Define the matrices
\[ \Phi_n = \begin{bmatrix}
q^{n(1-2n)} & 0 &0 \\
0 & (-1)^nq^{n^2} &0\\
0&0&(-1)^nq^{n^2}
\end{bmatrix} ,\]
and,
\[ \Psi_n = \begin{bmatrix}
1 & 0 &0 \\
0 & q^{n(\beta-\alpha)} &0\\
0&0&q^{n(\alpha-\beta)}
\end{bmatrix} .\]
We make the scalar transformation
\begin{gather} \label{W transform}
W_n(t)=
\begin{cases}
\Psi_n\Phi_n
U_{n}(z)
\Phi_n^{-1},\qquad \text{for}\,z\in \text{ext}(\Gamma_{q^{2n}}),\\
\Psi_n\Phi_nU_n(t), \qquad  \text{for}\,z \in \text{int}(\Gamma_{q^{2n}}).
\end{cases}
\end{gather}

As a result $W_n(t)$ satisfies RHP IV:
\begin{enumerate}[label={{\rm (\roman *)}}]
\begin{subequations}
\item $W_n(t)$ is analytic on $\mathbb{C}\setminus \left( \Gamma_{q}\cup (-\infty,-I)\cup \{ q^{-k}\}_{k=1}^{2n}\right)$, where $-I$ is the minimum point where $\Gamma$ intersects the negative real axis.

\item $W_n(t)$ has continuous boundary values $W_{n,-}(s)$ and $W_{n,+}(s)$ as $t$ approaches $s\in\Gamma$ from $\mathcal D_-$ and $\mathcal D_+$ respectively, and satisfies the jump condition $W_{n,+}(s) = W_{n,-}(s)J_n(s)$ where
\begin{gather} \label{jump to ref}
J_n(s) = 
\begin{bmatrix}
A_n(s)^{-1} & s^\frac{1}{2}B_n(s)w(sq^{2n})h^\alpha(s) &C_n(s)w(sq^{2n})h^\beta(s) \\
0 & s^\frac{1}{2}B_n(s) &0\\
0&0&C_n(s)
\end{bmatrix},
\end{gather}
$w(s) = (sq;q)_\infty$ and $h_\alpha(s)$ and $h_\beta(s)$ are as discussed in Section \ref{Prelim sect}.

\item $W_n(t)$ has continuous boundary values $W_{n,-}(s)$ and $W_{n,+}(s)$ as $t$ approaches $s\in(-\infty,-I)$ from $x+i\epsilon$ and $x-i\epsilon$ respectively, and satisfies the jump condition $W_{n,+}(s) = W_{n,-}(s)M_n(s)$ where
\begin{gather} 
M_n(s) = 
\begin{bmatrix}
1&0&0 \\
0 & -1 &0\\
0&0&1
\end{bmatrix}
\end{gather}
(again, note that this jump condition is necessary to remedy the multiplication of the second column by $t^{1/2}$, it is not of importance to the rest of our argument).

\item $W_{n}(t)$ satisfies the residue condition
\begin{equation} 
 \mathrm{Res}_{q^{-2k},\,0<k\leq n}W_n(t)
 = 
  \lim_{t\to q^{-2k}}(t- q^{-2k}) W_n(t)
  \begin{bmatrix}
   0 &
   0 &
   0 \\
   \frac{1}{t^\frac{1}{2}B_n(t)A_n(t)w(tq^{2n})h^\alpha(t)} &
   0 &
   \frac{C_n(t)h^\beta(t)}{t^\frac{1}{2}B_n(t)h^\alpha(t)}\\
   0& 0 &0
   \end{bmatrix},
\end{equation}
\begin{equation}
\mathrm{Res}_{q^{-2k-1},\,0\leq k<n}W_n(t) = 
 \lim_{t\to q^{-2k-1}}(t- q^{-2k-1}) W_n(t)
  \begin{bmatrix}
   0 &
   0 &
   0 \\
   0 &
   0 &
   0\\
   \frac{1}{C_n(t)A_n(t))w(tq^{2n})h^\beta(t)} & \frac{t^\frac{1}{2}B_n(t)h^\alpha(t)}{C_n(t)h_\beta(t)} &0
   \end{bmatrix}.
\end{equation}

\item $W_n(t)$ satisfies the asymptotic condition
\begin{gather} 
W_n(t)\begin{pmatrix}
1 & 0 & 0\\
0 & q^{-n(1+\beta-\alpha)}t^{-\frac{1}{2}} & 0 \\
0 & 0 &q^{-n(\alpha-\beta)}
\end{pmatrix}
=
I + \mathcal{O}\left( \frac{1}{t} \right), \text{ as }\ |t| \rightarrow \infty.
\end{gather}
    
\end{subequations}
\end{enumerate}
Having made all these transformations and arrived at $W_n(t)$ we are now in position to prove the main result of this section. Note that after these transformations, $\text{det}(W_n(t)) = q^nt^{1/2}$. On the other hand, from Equation \eqref{weight tilde scaling} one immediately concludes that $q^{2n\alpha}\tilde{\gamma}_1^{(n-1,n)} = \gamma_1^{(n-1,n)}$ and $q^{2n\beta}\tilde{\gamma}_2^{(n,n-1)} = \gamma_2^{(n,n-1)}$, combining this with Lemmas \ref{Section 3 lemma} and \ref{Section 4 lemma} we find that in the limit $n\to\infty$:
\begin{eqnarray*}
    q^{n(1-2n)}P_{n,n}(tq^{2n}) &=& F_1(t),\\
 \frac{q^{n^2}q^{n(-\alpha+\beta)}}{\tilde{\gamma}_1^{(n-1,n)}} P_{n-1,n}(tq^{2n}) &=&  C_2 F_2(t),\\
    \frac{q^{n^2}q^{n(\alpha-\beta)}}{\tilde{\gamma}_2^{(n,n-1)}} P_{n,n-1}(tq^{2n}) &=&  C_3 F_2(t).
\end{eqnarray*}
Hence, although the determinant of $W_n(t)$ is of order $\mathcal{O}(q^{n})$ for finite $t\neq 0$ we have the benefit that each entry of $W_n(t)$ is of order $\mathcal{O}(1)$.\\

Suppose we have a solution to another RHP of the form of RHP I, but this time with weights which satisfy the condition given by Equation \eqref{comparable weights}. Let us call this solution $\mathcal{Y}_n(z)$.  After we perform the same transformations as we did to arrive at $W_n(t)$, this time to $\mathcal{Y}_n(z)$, we will in turn arrive at the corresponding RHP and accompanying matrix solution $\mathcal{W}_n(t)$. The RHP will be identical to RHP IV with the exception that $w(tq^{2n})$ will be replaced by the alternative weight, $\omega(tq^{2n})$ corresponding to $\mathcal{W}_n(t) $.\\

Let us modify $W_n(t)$ by altering the first column for $t\in \text{ext}(\Gamma)$, as follows,
\begin{equation}
    \widehat{W}_n(t) = W_n(t) + \sum_{k=1}^{2n} \frac{\omega(q^{2n-k})-w(q^{2n-k})}{t-q^{-k}} \text{Res}_{t=q^{-k}} \left( \begin{bmatrix}
q^{n(1-2n)}P_{n,n}(tq^{2n})\\
\frac{q^{n^2}q^{n(-\alpha+\beta)}}{\tilde{\gamma}_1^{(n-1,n)}} P_{n-1,n}(tq^{2n})\\
 \frac{q^{n^2}q^{n(\alpha-\beta)}}{\tilde{\gamma}_2^{(n,n-1)}} P_{n-1,n}(tq^{2n})
\end{bmatrix}/A_n(t)\right),
\end{equation}
where $A_n(t)$ is defined by Equation \eqref{induct diff}. Note that $P_{n,m}(tq^{2n})$ is an entire function and $A_n(t)$ has zeros at $t=q^k$ for $k>q^{-2n}$. We only consider the residues for $t=q^{-k}$ where $k\geq 1$, as we are interested in $t\in \text{ext}(\Gamma)$. By assumption $|\omega(q^{2n})-w(q^{2n})|=\mathcal{O}(q^{2n})$. Thus, combined with the vanishing residue of $P_{n\pm 1,n\pm 1}(t)$ detailed in Lemma \ref{Section 3 lemma}, we conclude that the difference between $W_n(t)$ and $\widehat{W}_n(t)$ will be of order $\mathcal{O}(q^{2n})$. Now let us consider 
\[ R(t) = \mathcal{W}_n(t)\widehat{W}_n(t)^{-1} .\]
$R(t)$ is a solution to the following RHP
\begin{enumerate}[label={{\rm (\roman *)}}]
\begin{subequations}
\item $R(t)$ is analytic on $\mathbb{C}\setminus  \Gamma_{q}$.

\item $R(t)$ has continuous boundary values $R_{n,-}(s)$ and $R_{n,+}(s)$ as $t$ approaches $s\in\Gamma$ from $\mathcal D_-$ and $\mathcal D_+$ respectively, and satisfies the jump condition $R_{n,+}(s) = R_{n,-}(s)J_n(s)$ where
\begin{eqnarray*} 
J_n(s)&=& (\mathcal{W}_{n,-}\widehat{W}_{n,-}^{-1})^{-1}(\mathcal{W}_{n,+}\widehat{W}_{n,+}^{-1}),\\
&=& \widehat{W}_{n,-}
(\mathcal{W}_{n,-}^{-1}\mathcal{W}_{n,+})\widehat{W}_{n,+}^{-1},\\
&=& \widehat{W}_{n,-}
(1+\mathcal{O}(q^{2n}))\widehat{W}_{n,-}^{-1},\\
&=&W_{n,-}\left( I + \mathcal{O}(q^{2n})\right) (W_{n,-} + \mathcal{O}(q^{2n}))^{-1}, \\
&=& I + \mathcal{O}(q^{n}).
\end{eqnarray*}
We have used Equation \eqref{jump to ref} to arrive at the third equality and the fact that $\det(W_{n,-})=\mathcal{O}(q^n)$ to arrive at the last equality. Note that as stated earlier, the elements of $W_n(t)$ are $\mathcal{O}(1)$ for finite $t$.
\item $R(t)$ satisfies the asymptotic condition
\begin{gather} 
R(t)
=
I + \mathcal{O}\left( \frac{1}{t} \right), \text{ as }\ |t| \rightarrow \infty.
\end{gather}
    
\end{subequations}
\end{enumerate}

As the jump is of size $I + \mathcal{O}(q^n)$ we can apply standard RHP analysis (see \cite{kuijlaars2003riemann},\cite{deift1993steepest}) to immediately conclude that $|R(t)-I| = \mathcal{O}(q^{n})$. Thus,
\[ \mathcal{W}_n(t) = W_n(t) + \mathcal{O}(q^{n}), \quad \text{as}\,n\to\infty.\]
This identity leads us to the following theorem which describes the asymptotic results found in this section.

\begin{theorem}\label{Section 5 lemma}
    Given a set of multiple $q$-orthogonal polynomials, $\{P_{n,m}(z)\}_{n,m=0}^\infty$, with weights given by Equation \eqref{the general weights}, where $\omega(x)$ satisfies Equation \eqref{the weight constraint}, both Lemmas \ref{Section 3 lemma} and \ref{Section 4 lemma} hold with the same limiting functions, $F_{i}(t)$, and constants, $C_i$.
\end{theorem}

\section{Results Discussion and Future Work}
In the main result of this paper, Theorem \ref{Main theorem}, we obtain  universal asymptotics as $n\to\infty$ for multiple $q$-orthogonal polynomials with weights which satisfy the constraint given by Equation \eqref{the weight constraint}. This result is hinted at in the work of \cite{postelmans2005multiple}. We can compare our results with \cite[Equation (2.7)]{postelmans2005multiple}, which tells us that multiple $q$-orthogonal polynomials with weights $w_1(x) = x^{\alpha_1} (qx;q)_\infty/(q^{\zeta+1}x;q)_\infty $ and $w_2(x) = x^{\alpha_1} (qx;q)_\infty/(q^{\zeta+1}x;q)_\infty $ satisfy the equation:
\begin{equation}\label{rod form}
p_{n,n}(z)(qz;q)_\infty/(q^{\zeta+1}z;q)_\infty = \aleph_n \sum_{k=0}^{\infty}\frac{\left(q^{-\zeta-2n};q\right)_{k}\left(q^{\alpha_1 + n +1};q\right)_{k}\left(q^{\alpha_2 + n +1};q\right)_{k}}{\left(q;q\right)_{k}\left(q^{\alpha_1+1};q\right)_{k}\left(q^{\alpha_2+1};q\right)_{k}}\*(q^{\zeta+1}z)^{k}.
\end{equation}
Writing $p_{n,n}(z)(qz;q)_\infty/(q^{\zeta+1}z;q)_\infty $ as a power series $\sum_{k=0}^\infty c_k z^k$, Equation \eqref{rod form} tells us
\[ c_k (1-q^k)(1-q^{\alpha_1+k})(1-q^{\alpha_2+k})  = c_{k-1}q^{\zeta+1}(1-q^{k-2n-\zeta-1})(1-q^{\alpha_1+k+n})(1-q^{\alpha_2+k+n}) .\]
If $y(z)=\sum_{0}^\infty c_kz^k$, then the above coefficient relation implies $y(z)$ satisfies the $q$-difference equation
\begin{eqnarray}\label{general q diff}
y(z)(1-zq^{\zeta+1}) &=& \left[ 1+q^{\alpha_1} + q^{\alpha_2}-qz(q^{-2n}+q^{n+\alpha_1+\zeta+1}+q^{n+\alpha_2+\zeta+1})\right]y(qz) \nonumber \\
&+& \left[-(q^{\alpha_1}+ q^{\alpha_2} + q^{\alpha_1+\alpha_2}) + q^2z(q^{\alpha_1}+q^{\alpha_2}+q^{\alpha_1+\alpha_2+3n+\zeta+1})q^{-n}\right]y(q^2z) \nonumber \\
&+& \left[ q^{\alpha_1+\alpha_2}-q^{3+\alpha_1+\alpha_2}z\right]y(q^3z),\label{rod diff}
\end{eqnarray}
One can observe that for fixed $\zeta$, the right hand side (RHS) of Equation \eqref{rod diff} approaches the same limit as $n\to\infty$ independent of $\zeta$. This seems to indicate that there are universal asymptotic results independent of the weight for weights which are `close enough'. This is exactly what we have just proved in this paper. These results are similar to those found in \cite{NJTLconst} for $q$-orthogonal polynomials with a single weight. However, there is one key difference. In the case of $q$-orthogonal polynomials with a single weight, the only value that mattered in the limiting behaviour of the polynomials as their degree tended to infinity was $\lim_{x\to0}\omega(x)$ (written in the notation of this paper). On the other hand, the proof that $\mathcal{W}_n(t)\to W_n(t)$ as $n\to\infty$ as detailed in this paper, breaks for weights $\omega(x)$ such that $|\omega(q^{2n})-1| > \mathcal{O}(q^n)$. It would be interesting to see if this is indeed the case and $\lim_{x\to0}\omega(x)$ is not the only determining factor for the asymptotic behaviour for multiple $q$-orthogonal polynomials. Equation \eqref{general q diff} indicates that for $\zeta << q^{-2n}$ the behaviour of $P_{n,n}(z)$ near $z=q^{2n}$ is independent of $\zeta$ to leading order. Thus, any difference in asymptotics for $\zeta =  \mathcal{O}(q^{-n}) $ will most likely show in the norms $\gamma^{(n\pm1,n\pm1)}_{1,2}$.

\subsection{Future work}
As mentioned in the introduction, a key reason behind pursuing a better understanding of the asymptotics of multiple $q$-orthogonal polynomials is their application to the $q$-zeta function described in \cite{POSTELMANS2007119}. More work is needed to apply the results presented here to the multiple $q$-orthogonal polynomials used in \cite{POSTELMANS2007119}. Interestingly, this future work involves applying the work of Adams \cite{Adams} and gaining a better understanding of $q$-difference equations with repeating eigenvalues (see \cite{Adams}) and their corresponding RHP. Work in this area would be a major contribution to our understanding of both $q$-difference equations and RHPs.\\

Another possible direction of research is extending the analysis presented here to type II multiple $q$-orthogonal polynomials, $P_{n,m}(z)$, with general $n$ and $m$. As discussed in Remark \ref{could generalise remark} the RHP is readily obtained for such a case, see \cite[Section 5]{BK} and \cite[Chapter 23.6]{ismail2005classical}, however, the asymptotic analysis is more detailed. We believe the techniques presented in this paper will extend to this generalisation, with more complexity to account for the extra degree of freedom for $m$. It is much more unclear how to extend our asymptotic analysis to multiple $q$-orthogonal polynomials of type I, see also \cite[Section 5]{BK}, which represents another direction of future research.

\section*{Disclosure Statement}
The authors report there are no competing interests to declare. Furthermore, there is no funding to declare.

\printbibliography
\end{document}